\documentclass[leqno,12pt]{article}

\def\today{${\scriptscriptstyle\number\day-\number\month-\number\year}$}
\usepackage{graphicx}
\usepackage{fancyhdr}
\usepackage{amsmath}
\usepackage{amsthm}
\usepackage{amsfonts}

\newtheorem{theorem}{Theorem}[section]
\newtheorem{lemma}[theorem]{Lemma}
\theoremstyle{definition}
\newtheorem{definition}[theorem]{Definition}
\newtheorem{remark}[theorem]{Remark}
\newtheorem{assumption}[theorem]{Assumption}
\newtheorem{notation}[theorem]{Notation}
\renewcommand{\theequation}{\thesection.\arabic{equation}}
\def\address#1{{\center{#1}}}
\date{}
\def\m@th{\mathsurround=0pt}
\def\eqal#1{\null\,\vcenter{\openup\jot\m@th
\ialign{\strut\hfil$\displaystyle{##}$&&$\displaystyle{{}##}$\hfil
 \crcr#1\crcr}}\,}
\def\matrix#1{\null\,\vcenter{\normalbaselines\m@th
 \ialign{\hfil$##$\hfil&&\quad\hfil$##$\hfil\crcr
 \mathstrut\crcr\noalign{\kern-\baselineskip}
 #1\crcr\mathstrut\crcr\noalign{\kern-\baselineskip}}}\,}

\def\N{{\Bbb N}}
\def\R{{\Bbb R}}

\def\divv{{\rm div}\,}
\def\rot{{\rm rot}\,}
\def\arctg{{\rm arctg}\,}
\def\tg{{\rm tg}\,}
\def\Im{{\rm Im}\,}
\def\\{\hfil\break}

\numberwithin{equation}{section}
\def\bye{\end{document}}
\title{Global regular axially symmetric solutions to the Navier-Stokes equations.\\ Part 1}
\author{Wojciech M. Zaj\c{a}czkowski}

\begin{document}
\input amssym.def
\input amssym.tex
\maketitle
\thispagestyle{fancy}

\address{Institute of Mathematics, Polish Academy of Sciences,\\
\'Sniadeckich 8, 00-656 Warsaw, Poland\\
e-mail:wz@impan.pl\\
Institute of Mathematics and Cryptology, 
Cybernetics Faculty, \\
Military University of Technology,\\
S. Kaliskiego 2, 00-908 Warsaw, Poland\\}

\begin{abstract}
The axially-symmetric solutions to the Navier-Stokes equations are considered in a bounded cylinder $\Omega\subset\R^3$ with the axis of symmetry. $S_1$ is the boundary of the cylinder parallel to the axis of symmetry and $S_2$ is perpendicular to it. We have two parts of $S_2$. For simplicity, we assume the periodic boundary conditions on $S_2$. On $S_1$ we impose vanishing of the normal component of velocity, angular component of velocity and the angular component of vorticity. We prove the existence of global regular solutions. To prove this we need that the coordinate of velocity along the axis of symmetry must vanish on it. We have to emphasize that the technique of weighted spaces applied to the stream function plays a crucial role in the proof of global regular axially symmetric solutions.
\end{abstract}


\section{Introduction}\label{s1}

In this paper we prove the existence of global regular axially-symmetric solutions to the Navier-Stokes equations in a cylindrical domain $\Omega\subset\R^3$:
$$
\Omega=\{x\in\R^3\colon x_1^2+x_2^2<R^2,|x_3|<a\},
$$
where $a$, $R$ are given positive numbers.
We denote by $x=(x_1,x_2,x_3)$ Cartesian coordinates. It is assumed that the $x_3$-axis is the axis of symmetry of $\Omega$.

Moreover,
$$\eqal{
&S_1=\{x\in\R^3\colon\sqrt{x_1^2+x_2^2}=R,x_3\in(-a,a)\},\cr
&S_2(a_0)=\{x\in\R^3\colon\sqrt{x_1^2+x^2}<R,x_3=a_0\in\{-a,a\}\},\cr}
$$
where $S_1$ is parallel to the axis of symmetry and $S_2(a_0)$ is perpendicular to it. $S_2(a_0)$ meets the axis of symmetry at $a_0$.

To describe the considered problem we introduce cylindrical coordinates $r$, $\varphi$, $z$ by the relations
\begin{equation}
x_1=r\cos\varphi,\quad x_2=r\sin\varphi,\quad x_3=z.
\label{1.1}
\end{equation}
The following orthonormal system
\begin{equation}
\bar e_r=(\cos\varphi,\sin\varphi,0),\ \ \bar e_\varphi=(-\sin\varphi,\cos\varphi,0),\ \ \bar e_z=(0,0,1)
\label{1.2}
\end{equation}
is connected with the cylindrical coordinates.

Any vector $u$ for the axially symmetric motions can be decomposed as follows
\begin{equation}
u=u_r(r,z,t)\bar e_r+u_\varphi(r,z,t)\bar e_\varphi+u_z(r,z,t)\bar e_z,
\label{1.3}
\end{equation}
where $u_r$, $u_\varphi$, $u_z$ are cylindrical coordinates of $u$.

Therefore, velocity $v$ and vorticity $\omega=\rot v$ are decomposed in the form
\begin{equation}
v=v_r(r,z,t)\bar e_r+v_\varphi(r,z,t)\bar e_\varphi+v_z(r,z,t)\bar e_z
\label{1.4}
\end{equation}
and
\begin{equation}
\omega=\omega_r(r,z,t)\bar e_r+\omega_\varphi(r,z,t)\bar e_\varphi+\omega_z(r,z,t)\bar e_z.
\label{1.5}
\end{equation}
The paper is devoted to a proof of global regular axially-symmetric solutions to the problem
\begin{equation}\eqal{
&v_{,t}+v\cdot\nabla v-\nu\Delta v+\nabla p=f\quad &{\rm in}\ \ \Omega^T=\Omega\times(0,T),\cr
&\divv v=0\quad &{\rm in}\ \ \Omega^T,\cr
&v\ \textrm{satisfies periodic boundary conditions}\quad &{\rm on}\ \  S_2^T=S_2\times(0,T),\cr
&v\cdot\bar n|_{S_1}=0,\ \ \omega_\varphi|_{S_1}=0,\ \ v_\varphi|_{S_1}=0\quad &{\rm on}\ \ S_1^T=S_1\times(0,T),\cr
&v|_{t=0}=v(0)\quad &{\rm in}\ \ \Omega,\cr}
\label{1.6}
\end{equation}
where $v=(v_1(x,t),v_2(x,t),v_3(x,t))\in\R^3$ is the velocity of the fluid,\break $p=p(x,t)\in\R$ is the pressure, $f=(f_1(x,t),f_2(x,t),f_3(x,t))\in\R^3$ is the external force field, $\nu>0$ is the constant viscosity coefficient.

Expressing problem (\ref{1.6}) in the cylindrical coordinates of velocity yields
\begin{equation}\eqal{
&v_{r,t}+v\cdot\nabla v_r-{v_\varphi^2\over r}-\nu\Delta v_r+\nu{v_r\over r^2}=-p_{,r}+f_r,\cr
&v_{\varphi,t}+v\cdot\nabla v_\varphi+{v_r\over r}v_\varphi-\nu\Delta v_\varphi+\nu{v_\varphi\over r^2}=f_\varphi,\cr
&v_{z,t}+v\cdot\nabla v_z-\nu\Delta v_z=-p_{,z}+f_z,\cr
&(rv_r)_{,r}+(rv_z)_{,z}=0\cr
&v_r|_{S_1}=0,\ \ v_\varphi|_{S_1}=0,\ \ v_{r,z}-v_{z,r}|_{S_1}=0,\cr
&v_r|_{t=0}=v_r(0),\ \ v_\varphi|_{t=0}=v_\varphi(0),\ \ v_z|_{t=0}=v_z(0),\cr}
\label{1.7}
\end{equation}
where we have the periodic boundary conditions on $S_2$ and
\begin{equation}\eqal{
&v\cdot\nabla=(v_r\bar e_r+v_z\bar e_z)\cdot\nabla=v_r\partial_r+v_z\partial_z,\cr
&\Delta u={1\over r}(ru_{,r})_{,r}+u_{,zz}.\cr}
\label{1.8}
\end{equation}
Formulating problem (\ref{1.1}) in terms of the cylindrical coordinates of vorticity implies
\begin{equation}\eqal{
&\omega_{r,t}+v\cdot\nabla\omega_r-\nu\Delta\omega_r+\nu{\omega_r\over r^2}=\omega_rv_{r,r}+\omega_zv_{r,z}+F_r,\cr
&\omega_{\varphi,t}+v\cdot\nabla\omega_\varphi-{v_r\over r}\omega_\varphi-\nu\Delta\omega_\varphi+\nu{\omega_\varphi\over r^2}={2\over r}v_\varphi v_{\varphi,z}+F_\varphi,\cr
&\omega_{z,t}+v\cdot\nabla\omega_z-\nu\Delta\omega_z= \omega_rv_{z,r}+\omega_zv_{z,z}+F_z,\cr
&\omega_r|_{t=0}=\omega_r(0),\ \ \omega_\varphi|_{t=0}=\omega_\varphi(0),\ \ \omega_z|_{t=0}=\omega(0)\cr}
\label{1.9}
\end{equation}
and we have boundary conditions $(\ref{1.7})_5$ on $S_1$ and the periodic boundary conditions on $S_2$, where $F=\rot f$ and
\begin{equation}
F=F_r(r,z,t)\bar e_r+F_\varphi(r,z,t)\bar e_\varphi+F_z(r,z,t)\bar e_z.
\label{1.10}
\end{equation}
The function
\begin{equation}
u=rv_\varphi
\label{1.11}
\end{equation}
is called swirl. It is a solution to the problem
\begin{equation}\eqal{
&u_{,t}+v\cdot\nabla u-\nu\Delta u+{2\nu\over r}u_{,r}=rf_\varphi\equiv f_0,\cr
&u|_{S_1}=0\ {\rm and}\ u\ \textrm{satisfies periodic boundary conditions on}\ S_2,\cr
&u|_{t=0}=u(0).\cr}
\label{1.12}
\end{equation}
The cylindrical components of vorticity can be described in terms of the cylindrical components of velocity and swirl in the following form
\begin{equation}\eqal{
&\omega_r=-v_{\varphi,z}=-{1\over r}u_{,z},\cr
&\omega_\varphi=v_{r,z}-v_{z,r},\cr
&\omega_z={1\over r}(rv_\varphi)_{,r}=v_{\varphi,r}+{v_\varphi\over r}={1\over r}u_{,r}.\cr}
\label{1.13}
\end{equation}
Equation $(\ref{1.7})_4$ implies existence of the stream function $\psi$ which is a solution to the problem
\begin{equation}\eqal{
&-\Delta\psi+{\psi\over r^2}=\omega_\varphi,\cr
&\psi|_{S_1}=0,\cr
&\psi\ \textrm{satisfies periodic boundary conditions on}\ S_2.\cr}
\label{1.14}
\end{equation}
Moreover, cylindrical components of velocity can be expressed in terms of the stream function in the following way
\begin{equation}\eqal{
&v_r=-\psi_{,z},\ \ v_z={1\over r}(r\psi)_{,r}=\psi_{,r}+{\psi\over r},\cr
&v_{r,r}=-\psi_{,zr},\ \ v_{r,z}=-\psi_{,zz},\cr
&v_{z,z}=\psi_{,rz}+{\psi_{,z}\over r},\ \ v_{z,r}=\psi_{,rr}+{1\over r}\psi_{,r}-{\psi\over r^2}.\cr}
\label{1.15}
\end{equation}
Introduce the pair
\begin{equation}
(\Phi,\Gamma)=(\omega_r/r,\omega_\varphi/r).
\label{1.16}
\end{equation}
Formula (\ref{1.6}) from \cite{CFZ} implies that quantities (\ref{1.16}) satisfy the following equations
\begin{equation}
\Phi_{,t}+v\cdot\nabla\Phi-\nu\bigg(\Delta+{2\over r}\partial_r\bigg)\Phi- (\omega_r\partial_r+\omega_z\partial_z){v_r\over r}=F_r/r\equiv\bar F_r
\label{1.17}
\end{equation}
and
\begin{equation}
\Gamma_{,t}+v\cdot\nabla\Gamma-\nu\bigg(\Delta+{2\over r}\partial_r\bigg)\Gamma+ 2{v_\varphi\over r}\Phi=F_\varphi/r\equiv\bar F_\varphi.
\label{1.18}
\end{equation}
We add the following initial and boundary conditions to solutions of (\ref{1.17}) and (\ref{1.18})
\begin{equation}\eqal{
&\Phi|_{S_1}=0,\ \ \Gamma|_{S_1}=0,\ \Phi,\ \Gamma\ \textrm{satisfy the periodic}\cr
&\textrm{boundary conditions on}\ S_2,\cr}
\label{1.19}
\end{equation}
\begin{equation}
\Phi|_{t=0}=\Phi(0),\ \ \Gamma|_{t=0}=\Gamma(0).
\label{1.20}
\end{equation}
Next, we express cylindrical coordinates of velocity in terms of $\psi_1=\psi/r$
\begin{equation}\eqal{
&v_r=-r\psi_{1,z},\ \ &v_z=(r\psi_1)_{,r}+\psi_1=r\psi_{1,r}+2\psi_1,\cr
&v_{r,r}=-\psi_{1,z}-r\psi_{1,rz},\ \ &v_{r,z}=-r\psi_{1,zz},\cr
&v_{z,z}=r\psi_{1,rz}+2\psi_{1,z},\ \ &v_{z,r}=3\psi_{1,r}+r\psi_{1,rr}.\cr}
\label{1.21}
\end{equation}
The aim of this paper is to prove the existence of global regular axially symmetric solutions to problem (\ref{1.6}). For this purpose we have to find a global estimate guaranteeing the existence of global regular solutions.

Function $\psi_1$ is a solution to the problem
\begin{equation}\eqal{
&-\Delta\psi_1-{2\over r}\psi_{1,r}=\omega_1\quad {\rm in}\ \ \Omega=(0,R)\times(-a,a),\cr
&\psi_1|_{r=R}=0,\cr
&\psi_1\ \textrm{satisfies the periodic boundary conditions on}\ S_2,\cr}
\label{1.22}
\end{equation}
where
\begin{equation}
\omega_1=\omega_\varphi/r.
\label{1.23}
\end{equation}
We have that $\omega_1=\Gamma$.

To state the main result we first introduce assumptions.

\begin{assumption}\label{a1.1}
Assume that the following quantities are finite:
$$\eqal{
&D_1=\|f\|_{L_2(\Omega^t)}+\|v(0)\|_{L_2(\Omega)},\cr
&D_2=\|f_0\|_{L_{\infty,1}(\Omega^t)}+\|u(0)\|_{L_\infty(\Omega)},\cr
&f_0=rf_\varphi,\ \ u=rv_\varphi,\cr
&D_3^2=D_1^2D_2^2+\|u_{,z}(0)\|_{L_2(\Omega)}^2+\|f_0\|_{L_2(\Omega^t)}^2,\cr
&D_4^2=D_1^2(1+D_2)+\|u_{,r}(0)\|_{L_2(\Omega)}^2+\|f_0\|_{L_2(\Omega^t)}^2+ \|f_0\|_{L_2(0,t;L_{4/3}(S_1))},\cr}
$$
where $D_1$, $D_2$ are introduced in (\ref{2.1}) and (\ref{2.7}), respectively, and $D_3$, $D_4$ in (\ref{5.2}) and (\ref{5.3}), respectively. Let
$$\eqal{
&D_5=D_2(D_1+D_2+D_3),\cr
&D_6=D_2^{1-\varepsilon_0}D_3,\cr}
$$
where $\varepsilon_0$ is arbitrary small positive number. Moreover,
$$\eqal{
D_7&=\|F_r\|_{L_2(0,t;L_{6/5}(\Omega))}^2+\|F_z\|_{L_2(0,t;L_{6/5}(\Omega))}^2\cr
&\quad+\|\omega_r(0)\|_{L_2(\Omega)}^2+\|\omega_z(0)\|_{L_2(\Omega)}^2\cr}
$$
is defined in Lemma \ref{l6.1}.

Next,
$$\eqal{
D_8&=\phi(D_2)(\|\bar F_r\|_{L_2(0,t;L_{6/5}(\Omega))}^2+\|\bar F_\varphi\|_{L_2(0,t;L_{6/5}(\Omega))}^2)\cr
&\quad+\|\Phi(0)\|_{L_2(\Omega)}^2+\|\Gamma(0)\|_{L_2(\Omega)}^2,\cr}
$$
where $\bar F_r=F_r/r$, $\bar F_\varphi=F_\varphi/r$, $\Phi={\omega_r\over r}$, $\Gamma={\omega_\varphi\over r}$ and $D_8$ appears in (\ref{4.1}).
\end{assumption}

In Lemma \ref{l4.5} the following quantity is defined
$$
D_9(12)=12\|f_\varphi\|_{L_{12}(0,t;L_{36/25}(\Omega))}+\|v_\varphi(0)\|_{L_{12}(\Omega)}.
$$
Finally, in Lemma \ref{l4.7} we have introduced the quantity
$$
D_{10}=\|f_\varphi/r\|_{L_1(0,t;L_\infty(\Omega))}+ \|v_\varphi(0)\|_{L_\infty(\Omega)}.
$$
The main result

\begin{theorem}\label{t1.2}
Assume that Assumption \ref{a1.1} holds. Then there exists an increasing positive function $\phi$ such that
\begin{equation}
\|\Phi\|_{V(\Omega^t)}+\|\Gamma\|_{V(\Omega^t)}\le\phi(D_1,\cdots,D_{10}).
\label{1.24}
\end{equation}
\end{theorem}

\begin{remark}\label{r1.3}
Estimate (\ref{1.24}) implies any regularity of solutions to problem (\ref{1.6}) assuming sufficient regularity of data.

To prove (\ref{1.24}) we need that $\psi_1$ and $v_z$ vanish on the axis of symmetry.

The proof of Theorem \ref{t1.2} is divided into the following steps:

\begin{itemize}
\item[1.] In Lemmas \ref{l2.2} and \ref{l2.3} we prove the energy estimate for solutions to (\ref{1.6}) and $L_\infty$-estimate for swirl.
\item[2.] In Lemma \ref{l2.5} the existence of weak solutions to problem (\ref{1.22}) for the stream function $\psi_1$ is proved for a given $\omega_1=\omega/r$. Solutions of (\ref{1.22}) have the form $\psi_1=\psi/r$. The weak solutions to (\ref{1.22}) proved in Lemma \ref{l2.5} do not vanish on the axis of symmetry.
\item[3.] In Section \ref{s3} for a given $\omega_1\in H^1(\Omega)$  many estimates for $\psi_1$ are found. In Lemma \ref{l3.3} we derived such estimate that $\psi_1$ must vanish on the axis of symmetry. We need the estimate in the proof of (\ref{1.24}). The result of Lemma \ref{l3.3} shows that weak solutions proved in Lemma~\ref{l2.5} must vanish on the axis of symmetry. In view of properties of the stream function it means that $v_z$ also vanishes on the axis of symmetry. In this section a theory of weighted Sobolev spaces from \cite{NZ} is used.
\item[4.] In Section \ref{s4} we were able to derive an estimate for
\begin{equation}
\|\Phi\|_{V(\Omega^t)}+\|\Gamma\|_{V(\Omega^t)}
\label{1.25}
\end{equation}
in terms of $\|v_\varphi\|_{L_\infty(0,t;L_{12}(\Omega))}$ and $\|v_\varphi\|_{L_\infty(\Omega^t)}^{\varepsilon_0}$, where $\varepsilon_0$ can be chosen as small as we want. The estimate holds in view of Lemma \ref{l6.1} and inequality (\ref{2.12}).
\item[5.] Finally, at the end of Section \ref{s4} and in Section \ref{s5} we were able to estimate $\|v_\varphi\|_{L_\infty(0,t;L_{12}(\Omega))}$ and $\|v_\varphi\|_{L_\infty(\Omega^t)}$ from the bound for (\ref{1.25}).
\end{itemize}
\end{remark}

The problem of regularity of axially-symmetric solutions to the Navier-Stokes equations has a long history. The first regularity results in the case of vanishing swirl are derived in \cite{L2} and \cite{UY} by O. A Ladyzhenskaya and Ukhovskii-Yudovich independently. Many references in the case of nonvanishing swirl can be found in \cite{NZ1}.

We have to emphasize that we were able to prove Theorem \ref{t1.2} because the theory of weighted Sobolev spaces developed in \cite{NZ} was used.

\section{Notation and auxiliary results}\label{s2}

First we introduce some notations

\begin{definition}\label{d2.1}
We use the following notation for Lebesque and Sobolev spaces
$$\eqal{
&\|u\|_{L_p(\Omega)}=|u|_{p,\Omega},\ \ \|u\|_{L_p(\Omega^t)}=|u|_{p,\Omega^t},\cr
&\|u\|_{L_{p,q}(\Omega^t)}=\|u\|_{L_q(0,t;L_p(\Omega))}=|u|_{p,q,\Omega^t},\cr}
$$
where $p,q\in[1,\infty]$. Next
$$\eqal{
&\|u\|_{H^s(\Omega)}=\|u\|_{s,\Omega},\ \ \|u\|_{W_p^s(\Omega)}=\|u\|_{s,p,\Omega},\cr
&\|u\|_{L_q(0,t;W_p^k(\Omega))}=\|u\|_{k,p,q,\Omega^t},\ \ \|u\|_{k,p,p,\Omega^t}=\|u\|_{k,p,\Omega^t},\cr}
$$
where $s,k\in\N\cup\{0\}$, $H^s(\Omega)=W_2^s(\Omega)$.
\end{definition}

We need energy type space $V(\Omega^t)$ appropriate for description of weak solutions to the Navier-Stokes equations
$$
\|u\|_{V(\Omega^t)}=|u|_{2,\infty,\Omega^t}+|\nabla u|_{2,\Omega^t}.
$$
We recall weighted Sobolev spaces defined by
$$
\|f\|_{H_\mu^k(\R_+)}= \bigg(\intop_{\R_+}\sum_{j=0}^k|\partial_r^jf|^2r^{2(\mu+j-k)}rdr\bigg)^{1/2}
$$
and
$$
\|f\|_{H_\mu^k(\Omega)}=\bigg(\intop_\Omega\sum_{|\alpha|=0}^k|D_{r,z}^\alpha f|^2r^{2(\mu+|\alpha|-k)}rdrdz\bigg)^{1/2},
$$
where $\Omega$ contains the axis of symmetry, $D^\alpha=\partial_r^{\alpha_1}\partial_z^{\alpha_2}$, $|\alpha|=\alpha_1+\alpha_2$, $\alpha_i\in\N\cup\{0\}$, $i=1,2$, $k\in\N\cup\{0\}$, $\mu\in\R_+$. Moreover, we have
$$\eqal{
&H_0^0(\Omega)=L_{2,0}(\Omega)=L_2(\Omega),\cr
&H_\mu^0(\Omega)=L_{2,\mu}(\Omega)\cr}
$$
and
$$
\|f\|_{L_{2,\mu}(\Omega)}=|f|_{2,\mu,\Omega}.
$$

\begin{lemma}\label{l2.2}
Let $f\in L_{2,1}(\Omega^t)$, $v(0)\in L_2(\Omega)$. Then solutions to (\ref{1.7}) satisfy the estimate
\begin{equation}\eqal{
&\|v(t)\|_{L_2(\Omega)}^2+\nu\intop_{\Omega^t}(|\nabla v_r|^2+|\nabla v_\varphi|^2+|\nabla v_z|^2)dxdt'\cr
&\quad+\nu\intop_{\Omega^t}\bigg({v_r^2\over r^2}+{v_z^2\over r^2}\bigg)dxdt'\le 3\|f\|_{L_{2,1}(\Omega^t)}^2+2\|v(0)\|_{L_2(\Omega)}^2\equiv D_1^2.\cr}
\label{2.1}
\end{equation}
\end{lemma}

\begin{proof}
Multiplying $(\ref{1.7})_1$ by $v_r$, $(\ref{1.7})_2$ by $v_\varphi$, $(\ref{1.7})_3$ by $v_z$, adding the results and integrating over $\Omega$ yield
\begin{equation}\eqal{
&{1\over 2}{d\over dt}\intop_\Omega(v_r^2+v_\varphi^2+v_z^2)dx+\nu\intop_\Omega(|\nabla v_r|^2+|\nabla v_\varphi|^2+|\nabla v_z|^2)dx\cr
&\quad+\nu\intop_\Omega\bigg({v_r^2\over r^2}+{v_\varphi^2\over r^2}\bigg)dx+ \intop_\Omega(p_{,r}v_r+p_{,z}v_z)dx\cr
&=\intop_\Omega(f_rv_r+f_\varphi v_\varphi+f_zv_z)dx.\cr}
\label{2.2}
\end{equation}
The last term on the l.h.s. of (\ref{2.2}) vanishes in virtue of the equation of continuity $(\ref{1.7})_4$ and boundary conditions.

Using that $v^2=v_r^2+v_\varphi^2+v_z^2$, (\ref{2.2}) takes the form
\begin{equation}\eqal{
&{1\over 2}{d\over dt}\|v\|_{L_2(\Omega)}^2+\nu\intop_\Omega(|\nabla v_r|^2+|\nabla v_\varphi|^2+|\nabla v_z|^2)dx\cr
&\quad+\nu\intop_\Omega\bigg({v_r^2\over r^2}+{v_\varphi^2\over r^2}\bigg)dx= \intop_\Omega(f_rv_r+f_\varphi v_\varphi+f_zv_z)dx.\cr}
\label{2.3}
\end{equation}
Applying the H\"older inequality to the r.h.s. of (\ref{2.3}) yields
\begin{equation}
{d\over dt}\|v\|_{L_2(\Omega)}\le\|f\|_{L_2(\Omega)},
\label{2.4}
\end{equation}
where $f^2=f_r^2+f_\varphi^2+f_z^2$.

Integrating (\ref{2.4}) with respect to time gives
\begin{equation}
\|v\|_{L_2(\Omega)}\le\|f\|_{L_{2,1}(\Omega^t)}+\|v(0)\|_{L_2(\Omega)}.
\label{2.5}
\end{equation}
Integrating (\ref{2.3}) with respect to time, using the H\"older inequality in the r.h.s. of (\ref{2.3}) and exploiting (\ref{2.5}), we obtain
\begin{equation}\eqal{
&{1\over 2}\|v(t)\|_{L_2(\Omega)}^2+\nu\intop_{\Omega^t}(|\nabla v_r|^2+|\nabla v_\varphi|^2+|\nabla v_z|^2)dxdt'\cr
&\quad+\nu\intop_{\Omega^t}\bigg({v_r^2\over r^2}+{v_\varphi^2\over r^2}\bigg) dxdt'\le\|f\|_{L_{2,1}(\Omega^t)}(\|f\|_{L_{2,1}(\Omega^t)}\cr
&\quad+\|v(0)\|_{L_2(\Omega)})+{1\over 2}\|v(0)\|_{L_2(\Omega)}^2.\cr}
\end{equation}
The above inequality implies (\ref{2.1}). This concludes the proof.
\end{proof}

\begin{lemma}\label{l2.3}
Consider problem (\ref{1.12}). Assume that $f_0\in L_{\infty,1}(\Omega^t)$ and $u(0)\in L_\infty(\Omega)$. Then
\begin{equation}
\|u(t)\|_{L_\infty(\Omega)}\le\|f_0\|_{L_{\infty,1}(\Omega^t)}+ \|u(0)\|_{L_\infty(\Omega)}\equiv D_2.
\label{2.7}
\end{equation}
\end{lemma}

\begin{proof}
Multiplying $(\ref{1.12})_1$ by $u|u|^{s-2}$, $s>2$, integrating over $\Omega$ and by parts, we obtain
\begin{equation}\eqal{
&{1\over s}{d\over dt}\|u\|_{L_s(\Omega)}^s+{4\nu(s-1)\over s^2}\|\nabla|u|^{s/2}\|_{L_2(\Omega)}^2+{\nu\over s}\intop_\Omega(|u|^s)_{,r}drdz\cr
&=\intop_\Omega f_0u|u|^{s-2}dx.\cr}
\label{2.8}
\end{equation}
From \cite{LW} it follows that $u|_{r=0}=0$. Moreover, using boundary conditions, (\ref{2.8}) implies
\begin{equation}
{d\over dt}\|u\|_{L_s(\Omega)}\le\|f_0\|_{L_s(\Omega)}.
\label{2.9}
\end{equation}
Integrating (\ref{2.9}) with respect to time and passing with $s\to\infty$, we derive (\ref{2.7}). This ends the proof.
\end{proof}

\begin{lemma}\label{l2.4}
Let estimates (\ref{2.1}) and (\ref{2.7}) hold. Then
\begin{equation}
\|v\|_{L_4(\Omega^t)}\le D_1^{1/2}D_2^{1/2}.
\label{2.10}
\end{equation}
\end{lemma}

\begin{proof}
We have
$$
\intop_{\Omega^t}|v_\varphi|^4dxdt'=\intop_{\Omega^t}r^2v_\varphi^2{v_\varphi^2\over r^2}dxdt'\le\|rv_\varphi\|_{L_\infty(\Omega^t)}^2\intop_{\Omega^t)}{v_\varphi^2\over r^2}dxdt'\le D_2^2D_1^2.
$$
This implies (\ref{2.10}) and concludes the proof.
\end{proof}

\begin{lemma}\label{l2.5}
Consider problem (\ref{1.22}). Assume that $\omega_1\in L_{6/5}(\Omega)$, where $\Omega=(0,R)\times(-a,a)$. Then there exists a weak solution to problem (\ref{1.22}) such that $\psi_1\in H^1(\Omega)$ and the estimate
\begin{equation}
\|\psi_1\|_{1,\Omega}\le c|\omega_1|_{6/5,\Omega}
\label{2.11}
\end{equation}
holds.
\end{lemma}

\begin{proof}
Multiplying $(\ref{1.22})_1$ by $\psi_1$ and using the boundary conditions we obtain
$$
\|\psi_1\|_{1,\Omega}^2+\intop_{-a}^a\psi_1^2|_{r=R}dz=\intop_\Omega\omega_1\psi_1dx.
$$
Applying the H\"older and Young inequality to the r.h.s. implies (\ref{2.11}). The Fredholm theorem gives existence. This ends the proof.
\end{proof}

\begin{remark}\label{r2.6}
We have to emphasize that the weak solution $\psi_1$ of (\ref{1.22}) does not vanish on the axis of symmetry. It also follows from \cite{LW}.
\end{remark}

From Lemma 2.4 in \cite{CFZ} we also have

\begin{lemma}\label{l2.7}
Let $f\in C^\infty((0,R)\times(-a,a))$, $f|_{r\ge R}=0$. Let $1<r\le 3$, $0\le s\le r$, $s\le 2$, $q\in\big[r,{r(3-s)\over 3-r}\big]$. Then there exists a positive constant $c=c(s,r)$ such that
\begin{equation}
\bigg(\intop_\Omega{|f|^q\over r^s}dx\bigg)^{1/q}\le c|f|_{r,\Omega}^{{3-s\over q}-{3\over r}+1}|\nabla f|_{r,\Omega}^{{3\over r}-{3-s\over q}},
\label{2.12}
\end{equation}
where $f$ does not depend on $\varphi$.
\end{lemma}

\begin{notation}[see \cite{NZ}]

First we introduce the Fourier transform. Let $f\in S(\R)$, where $S(\R)$ is the Schwartz space of all complex-valued rapidly decreasing infinitely differentiable functions on $\R$. Then the Fourier transform of $f$ and its inverse are defined by
\begin{equation}
\hat f(\lambda)={1\over\sqrt{2\pi}}\intop_\R e^{-i\lambda\tau}f(\tau)d\tau,\quad \check{\hat f}_{(\tau)}={1\over\sqrt{2\pi}}\intop_\R e^{i\lambda\tau}\hat f(\lambda)d\lambda
\label{2.13}
\end{equation}
and $\check{\hat f}=\hat{\check f}=f$.

By $H_\mu^k(\R_+)$ we denote a weighted space with the norm
$$
\|u\|_{H_\mu^k(\R_+)}=\sum_{i=0}^k\intop_{\R_+}|\partial_r^iu|^2r^{2(\mu-k+i)}rdr.
$$
In view of transformation $\tau=-\ln r$, $r=e^{-\tau}$, $dr=-e^{-\tau}d\tau$ we have the equivalence
\begin{equation}
\sum_{i=0}^k\intop_{\R_+}|\partial_r^iu|^2r^{2(\mu-k+i)}rdr\sim\sum_{i=0}^k \intop_\R|\partial_\tau^iu'|^2e^{2h\tau}d\tau
\label{2.14}
\end{equation}
which holds for $u'(\tau)=u'(-\ln r)=u(r)$, $h=k+1-\mu$.
\end{notation}

In view of the Fourier tranform (\ref{2.13}) and the Parseval identity we have
\begin{equation}
\intop_{-\infty+ih}^{+\infty+ih}\sum_{j=0}^k|\lambda|^{2j}|\hat u(\lambda)|^2d\lambda=\intop_\R\sum_{j=0}^k|\partial_\tau^ju|^2e^{2h\tau}d\tau.
\label{2.15}
\end{equation}
 
\section{Estimates for the stream function $\psi_1$}\label{s3}

Recall that $\psi_1$ is a solution to the problem
\begin{equation}\eqal{
&-\psi_{1,rr}-\psi_{1,zz}-{3\over r}\psi_{1,r}=\omega_1\quad &{\rm in}\ \ \Omega=(0,R)\times(-a,a),\cr
&\psi_1|_{r=R}=0,\cr
&\psi_1\ \textrm{satisfies the periodic boundary conditions}\quad &{\rm on}\ \ S_2.\cr}
\label{3.1}
\end{equation}

\begin{lemma}\label{l3.1}
For sufficiently regular solutions to (\ref{3.1}) the following estimates hold
\begin{equation}\eqal{
&\intop_\Omega(\psi_{1,rr}^2+\psi_{1,rz}^2+\psi_{1,zz}^2)dx+\intop_\Omega{1\over r^2}\psi_{1,r}^2dx+\intop_{-a}^a\psi_{1,z}^2|_{r=0}dz\cr
&\quad+\intop_{-a}^a\psi_{1,r}^2|_{r=R}dz\le c|\omega_1|_{2,\Omega}^2\cr}
\label{3.2}
\end{equation}
and
\begin{equation}
\intop_\Omega(\psi_{1,zzr}^2+\psi_{1,zzz}^2)dx+\intop_{-a}^a\psi_{1,zz}^2|_{r=0}dz\le c|\omega_{1,z}|_{2,\Omega}^2
\label{3.3}
\end{equation}
and
\begin{equation}\eqal{
&\intop_\Omega(\psi_{1,rrz}^2+\psi_{1,rzz}^2+\psi_{1,zzz}^2)dx+\intop_{-a}^a \psi_{1,zz}^2|_{r=0}dz\cr
&\quad+\intop_{-a}^a\psi_{1,rz}^2|_{r=R}dz\le c|\omega_{1,z}|_{2,\Omega}^2.\cr}
\label{3.4}
\end{equation}
\end{lemma}

\begin{proof}
First we prove (\ref{3.2}). Multiplying $(\ref{3.1})_1$ by $\psi_{1,zz}$ and integrating over $\Omega$ yields
\begin{equation}
-\intop_\Omega\psi_{1,rr}\psi_{1,zz}dx-\intop_\Omega\psi_{1,zz}^2dx-3\intop_\Omega{1\over r}\psi_{1,r}\psi_{1,zz}dx=\intop_\Omega\omega_1\psi_{1,zz}dx.
\label{3.5}
\end{equation}
Integrating by parts with respect to $r$ in the first term implies
$$\eqal{
&-\intop_\Omega(\psi_{1,r}\psi_{1,zz}r)_{,r}drdz+\intop_\Omega\psi_{1,r}\psi_{1,zzr}dx+ \intop_\Omega\psi_{1,r}\psi_{1,zz}drdz\cr
&\quad-\intop_\Omega\psi_{1,zz}^2dx-3\intop_\Omega\psi_{1,r}\psi_{1,zz}drdz= \intop_\Omega\omega_1\psi_{1,zz}dx.\cr}
$$
Continuing, we get
\begin{equation}\eqal{
&-\intop_{-a}^a\psi_{1,r}\psi_{1,zz}r\bigg|_{r=0}^{r=R}dz+\intop_\Omega\psi_{1,r} \psi_{1,zzr}dx-\intop_\Omega\psi_{1,zz}^2dx\cr
&\quad-2\intop_\Omega\psi_{1,r}\psi_{1,zz}drdz=\intop_\Omega\omega_1\psi_{1,zz}dx.\cr}
\label{3.6}
\end{equation}
The first integral in (\ref{3.6}) vanishes because $\psi_{1,r}r|_{r=0}=0$, $\psi_{1,zz}|_{r=R}=0$. Integrating by parts with respect to $z$ in the last term on the l.h.s. of (\ref{3.6}) and using the periodic boundary conditions on $S_2$ we obtain
\begin{equation}
\intop_\Omega\psi_{1,r}\psi_{1,zzr}dx-\intop_\Omega\psi_{1,zz}^2dx+2\intop_\Omega \psi_{1,rz}\psi_{1,z}drdz=\intop_\Omega\omega_1\psi_{1,zz}dx.
\label{3.7}
\end{equation}
Integrating by parts with respect to $z$ in the first term in (\ref{3.7}) and using the boundary conditions on $S_2$ we get
\begin{equation}
\intop_\Omega(\psi_{1,zr}^2+\psi_{1,zz}^2)dx-\intop_\Omega(\psi_{1,z}^2)_{,r}drdz= -\intop_\Omega\omega_1\psi_{1,zz}dx,
\label{3.8}
\end{equation}
where the last term on the l.h.s. equals
$$
-\intop_{-a}^a\psi_{1,z}^2\bigg|_{r=0}^{r=R}dz=\intop_{-a}^a\psi_{1,z}^2\bigg|_{r=0}dz
$$
because $\psi_{1,z}|_{r=R}=0$. Using this in (\ref{3.8}) and applying the H\"older and Young inequalities to the r.h.s. of (\ref{3.8}) yield
\begin{equation}
\intop_\Omega(\psi_{1,rz}^2+\psi_{1,zz}^2)dx+\intop_{-a}^a\psi_{1,z}^2|_{r=0}dz\le c|\omega_1|_{2,\Omega}^2.
\label{3.9}
\end{equation}
Multiply $(\ref{3.1})_1$ by ${1\over r}\psi_{1,r}$ and integrate over $\Omega$. Then we have
\begin{equation}
3\intop_\Omega\bigg|{1\over r}\psi_{1,r}\bigg|^2dx=-\intop_\Omega\psi_{1,rr}{1\over r}\psi_{1,r}dx-\intop_\Omega\psi_{1,zz}{1\over r}\psi_{1,r}dx-\intop_\Omega\omega_1{1\over r}\psi_{1,r}dx.
\label{3.10}
\end{equation}
The first term on the r.h.s. of (\ref{3.10}) equals
$$
-{1\over 2}\intop_\Omega\partial_r\psi_{1,r}^2drdz=-{1\over 2}\intop_{-a}^a\psi_{1,r}^2\bigg|_{r=0}^{r=R}dz=-{1\over 2}\intop_{-a}^a\psi_{1,r}^2|_{r=R}dz,
$$
because $\psi_{1,r}|_{r=0}=0$ (see \cite{LW}). Applying the H\"older and Young inequalities to the last two terms on the r.h.s. of (\ref{3.10}) implies
\begin{equation}
\intop_\Omega\bigg|{1\over r}\psi_{1,r}\bigg|^2dx+{1\over 2}\intop_{-a}^a\psi_{1,r}^2\bigg|_{r=R}dz\le c(|\psi_{1,zz}|_{2,\Omega}^2+|\omega_1|_{2,\Omega}^2).
\label{3.11}
\end{equation}
Inequalities (\ref{3.9}) and (\ref{3.11}) imply the estimate
\begin{equation}\eqal{
&\intop_\Omega(\psi_{1,rz}^2+\psi_{1,zz}^2)dx+\intop_\Omega\bigg|{1\over r}\psi_{1,r}\bigg|^2dx+\intop_{-a}^a\psi_{1,z}^2\bigg|_{r=0}dz\cr
&\quad+\intop_{-a}^a\psi_{1,r}^2\bigg|_{r=R}dz\le c|\omega_1|_{1,\Omega}^2.\cr}
\label{3.12}
\end{equation}
From $(\ref{3.1})_1$ we have
\begin{equation}
|\psi_{1,rr}|_{2,\Omega}^2\le|\psi_{1,zz}|_{2,\Omega}^2+3\bigg|{1\over r}\psi_{1,r}\bigg|_{2,\Omega}^2+|\omega_1|_{2,\Omega}^2.
\label{3.13}
\end{equation}
Inequalities (\ref{3.12}) and (\ref{3.13}) imply (\ref{3.2}).

Now, we show (\ref{3.3}). Differentiate $(\ref{3.1})_1$ with respect to $z$, multiply by $-\psi_{1,zzz}$ and integrate over $\Omega$. Then, we obtain
\begin{equation}
\intop_\Omega\!\psi_{1,rrz}\psi_{1,zzz}dx+\!\intop_\Omega\!\psi_{1,zzz}^2dx+3\!\intop_\Omega{1\over r}\!\psi_{1,rz}\psi_{1,zzz}dx=-\!\intop_\Omega\omega_{1,z}\psi_{1,zzz}dx.
\label{3.14}
\end{equation}
Integrating by parts with respect to $z$ yields
\begin{equation}
\intop_\Omega\psi_{1,rrz}\psi_{1,zzz}dx=\intop_\Omega(\psi_{1,rrz}\psi_{1,zz})_{,z}dx- \intop_\Omega\psi_{1,rrzz}\psi_{1,zz}dx,
\label{3.15}
\end{equation}
where the first integral vanishes in view of periodic boundary conditions on $S_2$. Integrating by parts with respect to $r$ in the second integral in (\ref{3.15}) gives
$$
-\intop_\Omega(\psi_{1,rzz}\psi_{1,zz}r)_{,r}drdz+\intop_\Omega\psi_{1,rzz}^2dx+ \intop_\Omega\psi_{1,rzz}\psi_{1,zz}drdz,
$$
where the first integral vanishes because
$$
\psi_{1,rzz}r|_{r=0}=0,\quad \psi_{1,zz}|_{r=R}=0.
$$
In view of the above considerations, (\ref{3.14}) takes the form
\begin{equation}\eqal{
&\intop_\Omega(\psi_{1,rzz}^2+\psi_{1,zzz}^2)dx+ \intop_\Omega\psi_{1,rzz}\psi_{1,zz}drdz\cr
&\quad+3\intop_\Omega\psi_{1,rz}\psi_{1,zzz}drdz= -\intop_\Omega\omega_{1,z}\psi_{1,zzz}dx.\cr}
\label{3.16}
\end{equation}
Integrating by parts with respect to $z$ in the last term on the l.h.s. of (\ref{3.16}) and using the periodic boundary conditions on $S_2$ we get
\begin{equation}\eqal{
&\intop_\Omega(\psi_{1,rzz}^2+\psi_{1,zzz}^2)dx- \intop_\Omega\partial_r\psi_{1,zz}^2drdz\cr
&=-\intop_\Omega\omega_{1,z}\psi_{1,zzz}dx.\cr}
\label{3.17}
\end{equation}
Applying the H\"older and Young inequalities to the r.h.s. of (\ref{3.17}) yields
$$
\intop_\Omega(\psi_{1,rzz}^2+\psi_{1,zzz}^2)dx+ \intop_{-a}^a\psi_{1,zz}^2\bigg|_{r=0}dz\le c|\omega_{1,z}|_{2,\Omega}^2,
$$
where we used that $\psi_{1,zz}|_{r=R}=0$.

The above inequality implies (\ref{3.3}).

Finally, we show (\ref{3.4}). Differentiate $(\ref{3.1})_1$ with respect to $z$, multiply by $\psi_{1,rrz}$ and integrate over $\Omega$. Then we have
\begin{equation}\eqal{
&-\intop_\Omega\psi_{1,rrz}^2dx-\intop_\Omega\psi_{1,zzz}\psi_{1,rrz}dx-3\intop_\Omega {1\over r}\psi_{1,rz}\psi_{1,rrz}dx\cr
&=\intop_\Omega\omega_{1,z}\psi_{1,rrz}dx.\cr}
\label{3.18}
\end{equation}
Integrating by parts with respect to $z$ in the second term in (\ref{3.18}) implies
$$\eqal{
&-\intop_\Omega\psi_{1,zzz}\psi_{1,rrz}dx=\intop_\Omega\psi_{1,zz}\psi_{1,rrzz}dx= \intop_\Omega(\psi_{1,zz}\psi_{1,rzz}r)_rdrdz\cr
&\quad-\intop_\Omega\psi_{1,rzz}^2dx-\intop_\Omega\psi_{1,zz}\psi_{1,rzz}drdz,\cr}
$$
where the first term vanishes because
$$
\psi_{1,rzz}r|_{r=0}=0,\quad \psi_{1,zz}|_{r=R}=0.
$$
Then (\ref{3.18}) takes the form
\begin{equation}\eqal{
&\intop_\Omega(\psi_{1,rrz}^2+\psi_{1,rzz}^2)dx+\intop_\Omega\psi_{1,zz}\psi_{1,rzz}drdz\cr
&\quad+3\intop_\Omega\psi_{1,rz}\psi_{1,rrz}drdz=-\intop_\Omega\omega_{1,z} \psi_{1,rrz}dx.\cr}
\label{3.19}
\end{equation}
The second term in (\ref{3.19}) equals
$$
{1\over 2}\intop_{-a}^a\psi_{1,zz}^2\bigg|_{r=0}^{r=R}dz=-{1\over 2}\intop_{-a}^a\psi_{1,zz}^2\bigg|_{r=0}dz
$$
because $\psi_{1,zz}|_{r=R}=0$, and the last term on the l.h.s. of (\ref{3.19}) has the form
$$
{3\over 2}\intop_\Omega\partial_r\psi_{1,rz}^2drdz={3\over 2}\intop_{-a}^a\psi_{1,rz}^2\bigg|_{r=0}^{r=R}dz={3\over 2}\intop_{-a}^a\psi_{1,rz}^2\bigg|_{r=R}dz
$$
because $\psi_{1,rz}|_{r=0}=0$.

Using the above expressions in (\ref{3.19}) implies the equality
\begin{equation}\eqal{
&\intop_\Omega(\psi_{1,rrz}^2+\psi_{1,rzz}^2)dx-{1\over 2}\intop_{-a}^a\psi_{1,zz}^2\bigg|_{r=0}dz+{3\over 2}\intop_{-a}^a\psi_{1,rz}^2\bigg|_{r=R}dz\cr
&=-\intop_\Omega\omega_{1,z}\psi_{1,rrz}dx.\cr}
\label{3.20}
\end{equation}
Applying the H\"older and Young inequalities in the r.h.s. of (\ref{3.20}) gives
\begin{equation}\eqal{
&\intop_\Omega(\psi_{1,rrz}^2+\psi_{1,rzz}^2)dx-{1\over 2}\intop_{-a}^a\psi_{1,zz}^2\bigg|_{r=0}dz\cr
&\quad+{3\over 2}\intop_{-a}^a\psi_{1,rz}^2\bigg|_{r=R}dz\le c|\omega_{1,z}|_{2,\Omega}^2.\cr}
\label{3.21}
\end{equation}
Inequalities (\ref{3.21}) and (\ref{3.3}) imply (\ref{3.4}). This ends the proof.
\end{proof}

\begin{lemma}\label{l3.2}
For sufficiently regular solutions to (\ref{3.1}) the following inequality
\begin{equation}
\bigg|{1\over r}\psi_{1,rz}\bigg|_{2,\Omega}\le c|\omega_{1,z}|_{2,\Omega}
\label{3.22}
\end{equation}
holds.
\end{lemma}

\begin{proof}
Differentiating (\ref{3.1}) with respect to $z$ implies
\begin{equation}
-\psi_{1,rrz}-\psi_{1,zzz}-{3\over r}\psi_{1,rz}=\omega_{1,z}
\label{3.23}
\end{equation}
From (\ref{3.23}) we have
\begin{equation}
\bigg|{1\over r}\psi_{1,rz}\bigg|_{2,\Omega}\le c(|\psi_{1,rrz}|_{2,\Omega}+|\psi_{1,zzz}|_{2,\Omega}+|\omega_{1,z}|_{2,\Omega}).
\label{3.24}
\end{equation}
Using (\ref{3.4}) in (\ref{3.24}) yields (\ref{3.22}). This concludes the proof.
\end{proof}

Now we estimate $\big|{\psi_{1,zz}\over r}\big|_{2,\Omega}$.

\begin{lemma}\label{l3.3}
Let $\psi_1$ be such weak solution to problem (\ref{3.1}) that it vanishes on the axis of symmetry. Then such sufficiently regular solutions to problem (\ref{3.1}) satisfy the estimate
\begin{equation}
\intop_\Omega{\psi_{1,zz}^2\over r^2}dx+\intop_\Omega\bigg(\psi_{1,zrr}^2+ {\psi_{1,zr}^2\over r^2}+{\psi_{1,z}^2\over r^4}\bigg)dx\le c|\omega_{1,z}|_{2,\Omega}^2.
\label{3.25}
\end{equation}
\end{lemma}

\begin{proof}
Differentiating (\ref{3.1}) with respect to $z$ yields
\begin{equation}\eqal{
&-\Delta\psi_{1,z}-{3\over r}\psi_{1,zr}=\omega_{1,z},\cr
&\psi_{1,z}|_{S_1}=0,\cr
&\psi_{1,z}\ \textrm{satisfies periodic boundary conditions on}\ S_2.\cr}
\label{3.26}
\end{equation}
Applying Lemma 3.1 from \cite{NZ} to problem (\ref{3.26}) gives
\begin{equation}
\intop_\Omega\bigg(\psi_{1,zrr}^2+{\psi_{1,zr}^2\over r^2}+{\psi_{1,z}^2\over r^4}\bigg)dx\le c(|\omega_{1,z}|_{2,\Omega}^2+|\psi_{1,zzz}|_{2,\Omega}^2)\le c|\omega_{1,z}|_{2,\Omega}^2,
\label{3.27}
\end{equation}
where (\ref{3.3}) is used in the last inequality.

To examine solutions to (\ref{3.26}) we use the notation
\begin{equation}
u=\psi_{1,z}.
\label{3.28}
\end{equation}
Then (\ref{3.26}) takes the form

\begin{equation}\eqal{
&-\Delta u-{2\over r}u_{,r}=\omega_{1,z},\cr
&u|_{S_1}=0,\cr
&u\ \textrm{satisfies periodic boundary conditions on}\ S_2.\cr}
\label{3.29}
\end{equation}
Multiply $(\ref{3.28})_1$ by $ur^{-2}$, integrate over $\Omega$ and express the Laplacian operator in cylindrical coordinates. Then we have
\begin{equation}
-\intop_\Omega\bigg(u_{,rr}+{1\over r}u_{,r}+u_{,zz}\bigg)ur^{-2}dx-2\intop_\Omega{1\over r}u_{,r}ur^{-2}dx=\intop_\Omega\omega_{1,z}ur^{-2}dx.
\label{3.30}
\end{equation}
Integrating by parts with respect to $z$ in the third term under the first integral we obtain
\begin{equation}
\intop_\Omega{u_{,z}^2\over r^2}dx=\intop_\Omega\bigg(u_{,rr}+{3\over r}u_{,r}\bigg)ur^{-2}dx+\intop_\Omega\omega_{1,z}ur^{-2}dx.
\label{3.31}
\end{equation}
Applying the H\"older and Young inequalities to the r.h.s. integrals, using that 
$u=\psi_{1,z}$ and (\ref{3.27}), we derive
\begin{equation}
\intop_\Omega{\psi_{1,zz}^2\over r^2}dx\le c\intop_\Omega\bigg(\psi_{1,zrr}^2+{\psi_{1,zr}^2\over r^2}+{\psi_{1,z}^2\over r^4}\bigg)dx+c|\omega_{1,z}|_{2,\Omega}^2.
\label{3.32}
\end{equation}
Using (\ref{3.27}) in (\ref{3.32}) implies (\ref{3.25}). This concludes the proof.
\end{proof}

\begin{remark}\label{r3.4}
Lemma \ref{l3.3} is necessary in the proof of global regular axially-symmetric solutions to problem (\ref{1.6}). However, it imposes strong restrictions on solutions to (\ref{1.6}) because the condition $\psi_1|_{r=0}=0$ implies that $v_z|_{r=0}=0$. We do not know how to omit the restriction in the presented proof in this paper.
\end{remark}

\begin{lemma}\label{l3.5}
Let $\mu>0$ and $\omega_1\in H_\mu^1(\Omega)$. Then for sufficiently smooth solutions to (\ref{3.1}) the following estimate is valid
\begin{equation}
\intop_\Omega\bigg(\psi_{1,rrr}^2+{\psi_{1,rr}^2\over r^2}+{\psi_{1,r}^2\over r^4}\bigg)r^{2\mu}dx\le cR^{2\mu}\|\omega_1\|_{1,\Omega}^2.
\label{3.33}
\end{equation}
\end{lemma}

\begin{proof}
To prove the lemma we introduce a partition of unity $\{\zeta^{(i)}(r)\}_{i=1,2}$ such that
$$
\sum_{i=1}^2\zeta^{(i)}(r)=1
$$
and
$$\eqal{
&\zeta^{(1)}(r)=\begin{cases}
1&\ r\le r_0,\cr 0&\ r\ge r_0+\lambda,\cr\end{cases}\cr
&\zeta^{(2)}(r)=\begin{cases}
0&\ r\le r_0,\cr 1&\ r\ge r_0+\lambda,\cr\end{cases}\cr}
$$
where $r_0<R$ and $\zeta^{(i)}(r)$, $i=1,2$, are smooth functions.

Introduce the notation
\begin{equation}
\psi_1^{(i)}=\psi_1\zeta^{(i)},\quad \omega_1^{(i)}=\omega_1\zeta^{(i)},\ \ i=1,2.
\label{3.34}
\end{equation}
Then functions (\ref{3.34}) satisfy the equations
\begin{equation}\eqal{
&-\psi_{1,rr}^{(i)}-\psi_{1,zz}^{(i)}-{3\over r}\psi_{1,r}^{(i)}=-2\psi_{1,r}\dot\zeta^{(i)}-\psi_1\ddot\zeta^{(i)}-{3\over r}\psi_1\dot\zeta^{(i)}\cr
&\quad+\omega_1^{(i)}\equiv g^{(i)},\ \ i=1,2,\cr}
\label{3.35}
\end{equation}
where dot denotes derivative with respect to $r$.

First we consider the case $i=1$. Differentiating (\ref{3.35}) for $i=1$ with respect to $r$ yields
\begin{equation}
-\psi_{1,rrr}^{(1)}-\psi_{1,rzz}^{(1)}-{3\over r}\psi_{1,rr}^{(1)}+{3\over r^2}\psi_{1,r}^{(1)}=g_{,r}^{(1)}.
\label{3.36}
\end{equation}
Introduce the notation
\begin{equation}
v=\psi_{1,r}^{(1)},\quad f=g_{,r}^{(1)}.
\label{3.37}
\end{equation}
Then (\ref{3.36}) takes the form
\begin{equation}\eqal{
&-v_{,rr}-v_{,zz}-{3\over r}v_{,r}+{3\over r}v=f\quad {\rm in}\ \ \Omega_{r_0},\cr
&v|_{r=r_0}=0,\cr
&v|_{S_2}\ \textrm{satisfies periodic boundary conditions},\cr}
\label{3.38}
\end{equation}
where $\Omega_{r_0}=\{x\in\Omega\colon r\in(0,r_0),z\in(-a,a)\}$ and $r_0<R$.

Multiplying $(\ref{3.38})_1$ by $r^2$ yields
$$
-r^2v_{,rr}-3rv_{,r}+3v=r^2(f+v_{,zz})\equiv g(r,z)
$$
or equivalently
\begin{equation}
-r\partial_r(r\partial_rv)-2r\partial_rv+3v=g(r,z).
\label{3.39}
\end{equation}
Introduce the new variable
$$
\tau=-\ln r,\quad r=e^{-\tau}.
$$
Since $r\partial_r=-\partial_\tau$ equation (\ref{3.39}) takes the form
\begin{equation}
-\partial_\tau^2v+2\partial_\tau v+3v=g(e^{-\tau},z)\equiv g'(\tau,z).
\label{3.40}
\end{equation}
Applying the Fourier transform (\ref{2.13}) to (\ref{3.40}) gives
\begin{equation}
\lambda^2\hat v+2i\lambda\hat v+3\hat v=\hat g'.
\label{3.41}
\end{equation}
Looking for solutions to the algebraic equation
$$
\lambda^2+2i\lambda+3=0
$$
we see that it has two solutions 
$$
\lambda_1=-3i,\ \ \lambda_2=i.
$$
For $\lambda\not\in\{-3i,i\}$ we can write solutions to (\ref{3.41}) in the form
\begin{equation}
\hat v={1\over\lambda^2+2i\lambda+3}\hat g'\equiv R(\lambda)\hat g'.
\label{3.42}
\end{equation}
Since $R(\lambda)$ does not have poles on the line ${\rm \Im}\lambda=1-\mu=h$, $\mu\in(0,1)$ we can use Lemma 3.1 from \cite{NZ}. Then we obtain
\begin{equation}\eqal{
&\intop_{-\infty+ih}^{\infty+ih}\sum_{j=0}^2|\lambda|^{2(2-j)}|\hat v|^2d\lambda\le c\intop_{-\infty+ih}^{+\infty+ih}\sum_{j=0}^2|\lambda|^{2(2-j)}|R(\lambda)\hat g'|^2d\lambda\cr
&\le c\intop_{-\infty+ih}^{+\infty+ih}|\hat g'|^2d\lambda.\cr}
\label{3.43}
\end{equation}
By the Parseval identity inequality (\ref{3.43}) becomes
$$
\intop_\R\sum_{j=0}^2|\partial_\tau^jv|^2e^{2h\tau}d\tau\le c\intop_\R|g'|^2e^{2h\tau}d\tau.
$$
Passing to variable $r$ yields
$$
\sum_{j=0}^2\intop_{\R_+}|\partial_r^jv|^2r^{2(\mu+j-2)}rdr\le c\intop_{\R_+}|g|^2r^{2(\mu-2)}rdr.
$$
Using that $g=r^2(f+v_{,zz})$, we get
\begin{equation}
\sum_{j=0}^2\intop_{\R_+}|\partial_r^jv|^2r^{2(\mu+j-2)}rdr\le c\intop_{\R_+}|f+v_{,zz}|^2r^{2\mu}rdr.
\label{3.44}
\end{equation}
Recalling notation (\ref{3.37}) we derive from (\ref{3.44}) the inequality
\begin{equation}\eqal{
&\sum_{j=0}^2\intop_{\Omega}|\partial_r^j\psi_{1,r}^{(1)}|^2r^{2(\mu+j-2)}dx\le c\intop_{\Omega}|g_{,r}^{(1)}|^2r^{2\mu}dx\cr
&\quad+c\intop_{\Omega}|\psi_{1,rzz}|^2r^{2\mu}dx.\cr}
\label{3.45}
\end{equation}
In view of (\ref{3.3}),
\begin{equation}
|\psi_{1,rzz}|_{2,\Omega}\le c|\omega_{1,z}|_{2,\Omega}.
\label{3.46}
\end{equation}
The first term on the r.h.s. of (\ref{3.45}) can be estimated by
\begin{equation}
|g_{,r}^{(1)}|_{2,\mu,\Omega}\le c(|\psi_{1,rr}|_{2,\Omega}+|\psi_{1,r}|_{2,\Omega}+|\psi_1|_{2,\Omega}+ |\omega_{1,r}|_{2,\Omega}+|\omega_1|_{2,\Omega}).
\label{3.47}
\end{equation}
Lemma \ref{l3.1} and inequalities (\ref{3.45}), (\ref{3.46}) and (\ref{3.47}) imply
\begin{equation}\eqal{
&\intop_{\Omega}\bigg(|\psi_{1,rrr}^{(1)}|^2+{|\psi_{,rr}^{(1)}|^2\over r^2}+{|\psi_{,r}^{(1)}|^2\over r^4}\bigg)r^{2\mu}rdrdz\cr
&\quad+\intop_{\Omega}|\psi_{1,rzz}|^2dx\le c(|\omega_{1,r}|_{2,\Omega}^2+|\omega_1|_{2,\Omega}^2).\cr}
\label{3.48}
\end{equation}
Function $\psi_1^{(2)}$ is a solution to the problem
\begin{equation}\eqal{
&-\Delta\psi_1^{(2)}=-2\psi_{1,r}\dot\zeta^{(2)}-\psi_1\ddot\zeta^{(2)}+{2\over r}\psi_{1,r}^{(2)}\cr
&\quad-{3\over r}\psi_1\dot\zeta^{(2)}+\omega_1^{(2)}\quad &{\rm in}\ \ \bar\Omega_{r_0},\cr
&\psi_1^{(2)}|_{r=R}=0,\ \ \psi^{(2)}=0\quad &{\rm for}\ \ r\le r_0,\cr
&\psi^{(2)}\ \textrm{satisfies periodic boundary conditions}\quad &{\rm on}\ \ S_2,\cr}
\label{3.49}
\end{equation}
where $\bar\Omega_{r_0}=\{x\in\R^3\colon r_0\le r\le R,z\in(-a,a)\}$ and dot denotes the derivative with respect to $r$.

For solutions to (\ref{3.49}) the following estimate holds
\begin{equation}
\|\psi_1^{(2)}\|_{3,\Omega}\le c(\|\psi_{1,r}\|_{1,\Omega}+\|\psi_1\|_{1,\Omega}+ \|\omega_1^{(2)}\|_{1,\Omega})\le c\|\omega_1\|_{1,\Omega}.
\label{3.50}
\end{equation}
From (\ref{2.11}), (\ref{3.48}) and (\ref{3.50}) inequality (\ref{3.33}) follows. This ends the proof.
\end{proof}

\section{Estimates for $\Phi$ and $\Gamma$}\label{s4}

Let $\Omega=\{(r,z)\colon r\in(0,R),z\in(-a,a)\}$. Let $\Phi=\omega_r/r$, $\Gamma=\omega_\varphi/r$ and $\Phi$, $\Gamma$ are solutions to problem (\ref{1.17})--(\ref{1.20}).

\begin{lemma}\label{l4.1}
Assume that $\Phi(0),\Gamma(0)\in L_2(\Omega)$, $\bar F_r,\bar F_\varphi\in L_2(0,t;L_{6/5}(\Omega))$. Let $D_2$ be defined by (\ref{2.7}) and let
$$
I_3=\intop_{\Omega^t}\bigg|{v_\varphi\over r}\Phi\Gamma\bigg|dxdt'<\infty.
$$
Then
\begin{equation}\eqal{
&|\Phi(t)|_{2,\Omega}^2+|\Gamma(t)|_{2,\Omega}^2+\nu(\|\Phi\|_{1,2,\Omega^t}^2+ \|\Gamma\|_{1,2,\Omega^t}^2)\cr
&\le\phi(D_2)\bigg|\intop_{\Omega^t}{v_\varphi\over r}\Phi\Gamma dxdt'\bigg|+\phi(D_2)(|\bar F_r|_{6/5,2,\Omega^t}^2\cr
&\quad+|\bar F_\varphi|_{6/5,2,\Omega^t}^2)+|\Phi(0)|_{2,\Omega}^2+ |\Gamma(0)|_{2,\Omega}^2\cr
&\equiv\phi(D_2)I_3+D_8.\cr}
\label{4.1}
\end{equation}
\end{lemma}

\begin{proof}
Multiplying (\ref{1.17}) by $\Phi$ and integrating over $\Omega$ yield
\begin{equation}\eqal{
&{1\over 2}{d\over dt}|\Phi|_{2,\Omega}^2+|\nabla\Phi|_{2,\Omega}^2-\intop_{-a}^a \Phi\bigg|_{r=0}^{r=R}dz\cr
&=\intop_\Omega(\omega_r\partial_r+\omega_z\partial_z){v_r\over r}\Phi dx+\intop_\Omega\bar F_r\Phi dx.\cr}
\label{4.2}
\end{equation}
To derive the second term on the l.h.s. of (\ref{4.2}) we consider (\ref{1.17}) in
$$
\bar\Omega=\{x\in\R^3\colon r<R,z\in(-a,a),\varphi\in(0,2\pi)\}.
$$
Then by the Green theorem and boundary conditions we obtain the second term on the l.h.s. of (\ref{4.2}) on $\bar\Omega$. Using that all quantities in (\ref{4.2}) do not depend on $\varphi$ we can drop integration with respect to $\varphi$ and obtain (\ref{4.2}).

Using $\Phi|_{r=R}=0$ and (\ref{1.13}) we have
\begin{equation}\eqal{
&{1\over 2}{d\over dt}|\Phi|_{2,\Omega}^2+|\nabla\Phi|_{2,\Omega}^2\le \intop_\Omega(\omega_r\partial_r+\omega_z\partial_z){v_r\over r}\Phi dx+\intop_\Omega\bar F_r\Phi dx\cr
&\le\intop_\Omega\bigg(-v_{\varphi,z}\partial_r{v_r\over r}+{\partial_r(rv_\varphi)\over r}\partial_z{v_r\over r}\bigg)\Phi rdrdz+\intop_\Omega\bar F_r\Phi dx\cr
&=\intop_\Omega v_\varphi\bigg(\partial_z\partial_r{v_r\over r}\bigg)\Phi+\partial_r{v_r\over r}\partial_z\Phi\bigg)dx\cr
&\quad+\intop_\Omega\partial_r\bigg(rv_\varphi\partial_z{v_r\over r}\Phi\bigg)drdz-\intop_\Omega v_\varphi\bigg(\bigg(\partial_z\partial_r{v_r\over r}\bigg)\Phi+\partial_z{v_r\over r}\partial_r\Phi\bigg)dx\cr
&\quad+\intop_\Omega\bar F_r\Phi dx=\intop_{-a}^arv_\varphi\partial_z{v_r\over r}\Phi\bigg|_{r=0}^{r=R}dz+\intop_\Omega v_\varphi\bigg(\partial_r{v_r\over r}\partial_z\Phi-\partial_z{v_r\over r}\partial_r\Phi\bigg)dx\cr
&\quad+\intop_\Omega\bar F_r\Phi dx\equiv\intop_{-a}^arv_\varphi\partial_z{v_r\over r}\Phi\bigg|_{r=0}^{r=R}dz+I+\intop_\Omega\bar F_r\Phi dx,\cr}
\label{4.3}
\end{equation}
where using the periodic boundary conditions on $S_2$, the boundary term vanishes because $v_\varphi|_{r=R}=0$, $v_r|_{r=R}=0$, $\Phi|_{r=R}=0$ and
$$
\intop_{-a}^arv_\varphi\partial_z{v_r\over r}\Phi\bigg|_{r=0}dz=0
$$
because \cite{LW} implies the following expansions near the axis of symmetry
$$\eqal{
&v_\varphi=a_1(z,t)r+a_2(z,t)r^3+\cdots,\cr
&v_r=\bar a_1(z,t)r+\bar a_2(z,t)r^3+\cdots\cr}
$$
and $\Phi=-{v_{\varphi,z}\over r}$.

Finally,  $I\le I_1+I_2$, where
\begin{equation}\eqal{
&I_1\le\intop_\Omega\bigg|v_\varphi\partial_r{v_r\over r}\Phi_{,z}\bigg|dx,\cr
&I_2\le\intop_\Omega\bigg|v_\varphi\partial_z{v_r\over r}\Phi_{,r}\bigg|dx.\cr}
\label{4.4}
\end{equation}
Now, we estimate $I_1$ and $I_2$. Recall that ${v_r\over r}=-\psi_{1,z}$. Then
$$\eqal{
I_1&\le\intop_\Omega|v_\varphi\psi_{1,rz}\Phi_{,z}|dx=\intop_\Omega\bigg|rv_\varphi {\psi_{1,rz}\over r}\Phi_{,z}\bigg|dx\cr
&\le|rv_\varphi|_{\infty,\Omega}\bigg|{\psi_{1,rz}\over r}\bigg|_{2,\Omega}|\Phi_{,z}|_{2,\Omega}\equiv I_1^1.\cr}
$$
From (\ref{2.7}) and (\ref{3.22}) we have (recall that $\Gamma=\omega_1$)
\begin{equation}
I_1^1\le cD_2|\Gamma_{,z}|_{2,\Omega}|\Phi_{,z}|_{2,\Omega}.
\label{4.5}
\end{equation}
Similarly, we calculate
\begin{equation}\eqal{
I_2&\le\intop_\Omega|v_\varphi\psi_{1,zz}\Phi_{,r}|dx\le|rv_\varphi|_{\infty,\Omega} \bigg|{\psi_{1,zz}\over r}\bigg|_{2,\Omega}|\Phi_{,r}|_{2,\Omega}\cr
&\le cD_2|\Gamma_{,z}|_{2,\Omega}|\Phi_{,r}|_{2,\Omega},\cr}
\label{4.6}
\end{equation}
where (\ref{3.25}) is used.

Finally, the last term on the r.h.s. of (\ref{4.3}) is bounded by
\begin{equation}
\varepsilon|\Phi|_{6,\Omega}^2+c(1/\varepsilon)|\bar F_r|_{6/5,\Omega}^2.
\label{4.7}
\end{equation}
Using estimates (\ref{4.5})--(\ref{4.7}) in (\ref{4.3}), assuming that $\varepsilon$ is sufficiently small and applying the Poincar\'e inequality we obtain
\begin{equation}
{d\over dt}|\Phi|_{2,\Omega}^2+\|\Phi\|_{1,\Omega}^2\le cD_2|\Gamma_{,z}|_{2,\Omega}|\nabla\Phi|_{2,\Omega}+c|\bar F_r|_{6/5,\Omega}^2.
\label{4.8}
\end{equation}
Multiplying (\ref{1.18}) by $\Gamma$, integrating over $\Omega$, using the boundary conditions and explanation about applying the Green theorem appeared below (\ref{4.2}) we obtain
\begin{equation}\eqal{
&{1\over 2}{d\over dt}|\Gamma|_{2,\Omega}^2+|\nabla\Gamma|_{2,\Omega}^2-\intop_{-a}^a\Gamma^2 \bigg|_{r=0}^{r=R}dz\cr
&\le 2\bigg|\intop_\Omega{v_\varphi\over r}\Phi\Gamma dx\bigg|+\intop_\Omega\bar F_\varphi\Gamma dx.\cr}
\label{4.9}
\end{equation}
Using that $\Gamma|_{r=R}=0$, applying the H\"older and Young inequalities to the last term on the r.h.s. of (\ref{4.9}) and using the Poincar\'e inequality we derive
\begin{equation}
{d\over dt}|\Gamma|_{2,\Omega}^2+\|\Gamma\|_{1,\Omega}^2\le 2\intop_\Omega{v_\varphi\over r}\Phi\Gamma dx+c|\bar F_\varphi|_{6/5,\Omega}^2.
\label{4.10}
\end{equation}
From (\ref{4.8}) and (\ref{4.10}) we have
\begin{equation}\eqal{
&{d\over dt}(|\Phi|_{2,\Omega}^2+|\Gamma|_{2,\Omega}^2)+\|\Phi\|_{1,\Omega}^2+ \|\Gamma\|_{1,\Omega}^2\le\phi(D_2)\bigg|\intop_\Omega{v_\varphi\over r}\Phi\Gamma dxdt'\bigg|\cr
&\quad+\phi(D_2)(|\bar F_r|_{6/5,\Omega}^2+|\bar F_\varphi|_{6/5,\Omega}^2),\cr}
\label{4.11}
\end{equation}
where $\phi$ is an increasing positive function. Integrating (\ref{4.11}) with respect to time yields (\ref{4.1}). This ends the proof.
\end{proof}

\begin{lemma}\label{l4.2}
Let the assumptions of Lemma \ref{l6.1} hold.\\ 
Let $v_\varphi\in L_\infty(0,t;L_d(\Omega))$, $d>3$. Let $\theta=\big(1-{3\over d}\big)\varepsilon_1-{3\over d}\varepsilon_2>0$, $\varepsilon=\varepsilon_1+\varepsilon_2$. Let $\varepsilon_0>0$ be arbitrary small.\\
Then
\begin{equation}\eqal{
I_3&\le c|v_\varphi|_{d,\infty,\Omega^t}^\varepsilon[c_1(1+ |v_\varphi|_{\infty,\Omega^t}^{{1\over 2}\theta\varepsilon_0}) \|\Gamma\|_{1,2,\Omega^t}^{{1\over 2}\theta}\cr
&\quad+c_2]|\nabla\Phi|_{2,\Omega^t}^{1-\theta}|\nabla\Gamma|_{2,\Omega^t},\cr}
\label{4.12}
\end{equation}
where $c_1$, $c_2$ depending on $D_5$, $D_7$ are introduced in $L_1^4$ below.
\end{lemma}

\begin{proof}
We examine
$$\eqal{
I_3&=\intop_{\Omega^t}\bigg|rv_\varphi{\Phi\over r}{\Gamma\over r}\bigg|dxdt'\cr
&\le\intop_{\Omega^t}|rv_\varphi|^{1-\varepsilon}|v_\varphi|^\varepsilon \bigg|{\Phi\over r^{1-\varepsilon_1}}\bigg|\bigg|{\Gamma\over r^{1-\varepsilon_2}}\bigg|dxdt'=I_3^1,\cr}
$$
where $\varepsilon=\varepsilon_1+\varepsilon_2$ and $\varepsilon_i$, $i=1,2$, are positive numbers.

Using (\ref{2.7}) and applying the H\"older inequality in $I_3^1$ yields
$$\eqal{
I_3^1&\le D_2^{1-\varepsilon}\bigg(\intop_{\Omega^t}|v_\varphi|^{2\varepsilon} \bigg|{\Phi\over r^{1-\varepsilon_1}}\bigg|^2dxdt'\bigg)^{1/2}\bigg|{\Gamma\over r^{1-\varepsilon_2}}\bigg|_{2,\Omega^t}\cr
&\equiv D_2^{1-\varepsilon}L|\Gamma/r^{1-\varepsilon_2}|_{2,\Omega^t}\equiv I_3^2.\cr}
$$
By the Hardy inequality we obtain
\begin{equation}
\bigg\|{\Gamma\over r}\bigg\|_{L_{2,\varepsilon_2}(\Omega^t)}\le c\|\nabla\Gamma\|_{L_{2,\varepsilon_2}(\Omega^t)}\le cR^{\varepsilon_2}|\nabla\Gamma|_{2,\Omega^t}.
\label{4.13}
\end{equation}
Now, we estimate $L$,
$$\eqal{
L&=\bigg(\intop_0^t\intop_\Omega|v_\varphi|^{2\varepsilon}\bigg|{\Phi\over r^{1-\varepsilon_1}}\bigg|^2dxdt'\bigg)^{1/2}\cr
&\le\bigg[\intop_0^t|v_\varphi|_{2\varepsilon\sigma,\Omega}^{2\varepsilon}\bigg( \intop_\Omega\bigg|{\Phi\over r^{1-\varepsilon_1}}\bigg|^qdx\bigg)^{2/q}dt'\bigg]^{1/2}\equiv L_1,\cr}
$$
where $1/\sigma+1/\sigma'=1$, $q=2\sigma'$. Let $d=2\varepsilon\sigma$. Then
$$
\sigma'={d\over d-2\varepsilon}\quad {\rm so}\quad q={2d\over d-2\varepsilon}.
$$
Continuing,
$$
L_1\le\sup_t|v_\varphi|_{d,\Omega}^\varepsilon\bigg(\intop_0^t\bigg|{\Phi\over r^{1-\varepsilon_1}}\bigg|_{q,\Omega}^2dt'\bigg)^{1/2}\equiv L_1^1L_1^2.
$$
Now, we estimate the second factor $L_1^2$.

For this purpose we use Lemma \ref{l2.7} for $r=2$. Let ${s\over q}=1-\varepsilon_1$. Then $q\in[2,2(3-s)]$. Since $s=(1-\varepsilon_1)q$ we have the restriction $2<q<6-2s=6-2(1-\varepsilon_1)q$. Then
\begin{equation}
2<q\le{6\over 3-2\varepsilon_1}
\label{4.14}
\end{equation}
and ${6\over 3-2\varepsilon_1}>2$ for any $\varepsilon_1\in(0,1)$.

Hence, Lemma \ref{l2.7} implies
$$\eqal{
L_1^2&=\bigg(\intop_0^t\bigg|{\Phi\over r^{1-\varepsilon_1}}\bigg|_{q,\Omega}^2dt'\bigg)^{1/2}\cr
&\le c\bigg(\intop_0^t|\Phi|_{2,\Omega}^{2({3-s\over q}-{1\over 2})} |\nabla\Phi|^{2({3\over 2}-{3-s\over q})}dt'\bigg)^{1/2}\cr
&\le c|\Phi|_{2,\Omega^t}^{{3-s\over q}-{1\over 2}} |\nabla\Phi|_{2,\Omega^t}^{{3\over 2}-{3-s\over q}}\equiv L_1^3,\cr}
$$
where we used that for $\theta={3-s\over q}-{1\over 2}$ $1-\theta={3\over 2}-{3-s\over q}$ so the H\"older inequality can be applied.

Using (\ref{6.1}) in $L_1^3$, we have
$$\eqal{
L_1^3&\le c(D_5^{{1\over 2}\theta}|\nabla\Gamma|_{2,\Omega^t}^{{1\over 2}\theta}+ D_5^{{1\over 2}\theta}|v_\varphi|_{\infty,\Omega^t}^{{1\over 2}\theta\varepsilon_0}\|\Gamma\|_{1,2,\Omega^t}^{{1\over 2}\theta}+D_7^{{1\over 2}\theta})\cdot|\nabla\Phi|_{2,\Omega^t}^{1-\theta}\cr
&\equiv[c_1(1+|v_\varphi|_{\infty,\Omega^t}^{{1\over 2}\theta\varepsilon_0}) \|\Gamma\|_{1,2,\Omega^t}^{{1\over 2}\theta}+c_2] |\nabla\Phi|_{2,\Omega^t}^{1-\theta}\equiv L_1^4,\cr}
$$
where $c_1$, $c_2$ depend on $D_5$, $D_7$.

To justify the above inequality we have to know that the following inequalities hold
\begin{equation}
\theta={3-s\over q}-{1\over 2}>0
\label{4.15}
\end{equation}
and
\begin{equation}
1-\theta={3\over 2}-{3-s\over q}>0.
\label{4.16}
\end{equation}
Consider (\ref{4.15}). Using the form of $q$ and ${s\over q}$ we have
$$
{3\over q}-{s\over q}-{1\over 2}>0\quad {\rm so}\quad {3(d-2\varepsilon)\over 2d}-(1-\varepsilon_1)-{1\over 2}>0.
$$
Hence
$$
{3\over 2}-{3\over d}\varepsilon-1+\varepsilon_1-{1\over 2}>0\quad {\rm so}\quad \varepsilon_1-{3\over d}(\varepsilon_1+\varepsilon_2)>0.
$$
Therefore the following inequality
\begin{equation}
\bigg(1-{3\over d}\bigg)\varepsilon_1-{3\over d}\varepsilon_2>0
\label{4.17}
\end{equation}
holds for $d>3$ and $\varepsilon_2$ sufficiently small. Moreover, (\ref{4.17}) implies
\begin{equation}
\varepsilon_1>{3\over d}{d\over d-3}\varepsilon_2={3\over d-3}\varepsilon_2.
\label{4.18}
\end{equation}
To exmine (\ref{4.16}) we calculate
\begin{equation}
{3\over 2}-{3(d-2\varepsilon)\over 2d}+1-\varepsilon_1=1+{3\over d}\varepsilon-\varepsilon_1=1-\bigg(1-{3\over d}\bigg)\varepsilon_1+{3\over d}\varepsilon_2.
\label{4.19}
\end{equation}
Since (\ref{4.19}) must be positive we have the restriction
\begin{equation}
1+{3\over d}\varepsilon_2>\bigg(1-{3\over d}\bigg)\varepsilon_1
\label{4.20}
\end{equation}
Using (\ref{4.18}) in (\ref{4.20}) implies
$$
1+{3\over d}\varepsilon_2>{3\over d}\varepsilon_2
$$
so there is no contradiction.

Hence, we have
\begin{equation}\eqal{
&\theta=\bigg(1-{3\over d}\bigg)\varepsilon_1-{3\over d}\varepsilon_2,\cr
&1-\theta=1-\bigg(1-{3\over d}\bigg)\varepsilon_1+{3\over d}\varepsilon_2,\cr}
\label{4.21}
\end{equation}
where $d>3$.

Finally
$$
I_3\le c|v_\varphi|_{d,\infty,\Omega^t}^\varepsilon[c_1(1+ |v_\varphi|_{\infty,\Omega^t}^{{1\over 2}\theta\varepsilon_0})\|\Gamma\|_{1,2,\Omega^t}^{{1\over 2}\theta}+c_2] |\nabla\Phi|_{2,\Omega^t}^{1-\theta}\cdot|\nabla\Gamma|_{2,\Omega^t}.
$$
This implies (\ref{4.12}) and ends the proof.
\end{proof}

Introduce the quantity
\begin{equation}
X(t)=\|\Phi\|_{V(\Omega^t)}+\|\Gamma\|_{V(\Omega^t)}.
\label{4.22}
\end{equation}

\begin{lemma}\label{l4.3}
Let the assumptions of Lemmas \ref{l4.1} and \ref{l4.2} hold. Let $\theta=\big(1-{3\over d}\big)\varepsilon_1-{3\over d}\varepsilon_2$, $\varepsilon=\varepsilon_1+\varepsilon_2$.\\
Then
\begin{equation}
X^2\le c_0|v_\varphi|_{d,\infty,\Omega^t}^{4\varepsilon\over\theta}(1+ |v_\varphi|_{\infty,\Omega^t}^{2\varepsilon_0})+c_0| v_\varphi|_{d,\infty,\Omega^t}^{2\varepsilon\over\theta}+D_8,
\label{4.23}
\end{equation}
where $c_0=\phi(D_5,D_7)$.
\end{lemma}

\begin{proof}
In view of notation (\ref{4.22}) inequalities (\ref{4.1}) and (\ref{4.12}) imply
\begin{equation}\eqal{
X^2&\le c|v_\varphi|_{d,\infty,\Omega^t}^\varepsilon[c_1(1+ |v_\varphi|_{\infty,\Omega^t}^{{1\over 2}\theta\varepsilon_0})X^{1-{1\over 2}\theta}\cr
&\quad+c_2X^{1-\theta}]X+D_8\equiv\alpha_1X^{2-{1\over 2}\theta}+\alpha_2X^{2-\theta}+D_8.\cr}
\label{4.24}
\end{equation}
Applying the Young inequality in (\ref{4.24}) implies
$$
X^2\le c\alpha_1^{4\over\theta}+c\alpha_2^{2\over\theta}+D_8.
$$
This yields (\ref{4.23}) and concludes the proof.
\end{proof}

\begin{remark}\label{r4.4}
Consider exponents in (\ref{4.23}). Then
\begin{equation}
\delta={4\varepsilon\over\theta}={4\varepsilon\over(1-{3\over d})\varepsilon_1-{3\over d}\varepsilon_2},\quad \delta_0={2\varepsilon\over(1-{3\over d})\varepsilon_1-{3\over d}\varepsilon_2}.
\label{4.25}
\end{equation}
For $\varepsilon_2$ small we have
$$
\delta={4\over 1-{3\over d}}+\varepsilon_*,\quad \delta_0={2\over 1-{3\over d}}+\varepsilon_{0*},
$$
where $\varepsilon_*$, $\varepsilon_{0*}$ are positive number which can be chosen very small.

For $d=12$ it follows that
\begin{equation}
\delta={16\over 3}+\varepsilon_*,\quad \delta_0={8\over 3}+\varepsilon_{0*}.
\label{4.26}
\end{equation}
This ends the remark.
\end{remark}

\begin{lemma}\label{l4.5}
Assume that $\varepsilon_1>11\varepsilon_2$, $s>1$ and
$$
D_9^s(s)=s^2|f_\varphi|_{{3s\over 2s+1},s,\Omega^t}^s+|v_\varphi(0)|_{s,\Omega}^s<\infty.
$$
Then
\begin{equation}
|v_\varphi|_{12,\infty,\Omega^t}^6\le c|v_\varphi|_{\infty,\Omega^t}^{{6(3\varepsilon_1-\varepsilon_2)\over \varepsilon_1-11\varepsilon_2}\varepsilon_0}+\phi(D_5,D_7)+c(D_8+D_9^{12})
\label{4.27}
\end{equation}
\end{lemma}

\begin{proof}
Multiply $(\ref{1.7})_2$ by $v_\varphi|v_\varphi|^{s-2}$, integrate over $\Omega$ and exploit the relation ${v_r\over r}=-\psi_{1,z}$. Then, we obtain
\begin{equation}\eqal{
&{1\over s}{d\over dt}|v_\varphi|_{s,\Omega}^s+{4\nu(s-1)\over s^2}|\nabla|v_\varphi|^{s/2}|_{2,\Omega}^2=\intop_\Omega\psi_{1,z}|v_\varphi|^sdx\cr
&\quad+\intop_\Omega f_\varphi v_\varphi|v_\varphi|^{s-2}dx.\cr}
\label{4.28}
\end{equation}
Integrating by parts in the first term on the r.h.s. of (\ref{4.28}) and applying the H\"older and Young inequalities yield
$$
\bigg|\intop_\Omega\psi_{1,z}|v_\varphi|^sdx\bigg|\le\varepsilon\big|\partial_z|v_\varphi|^{s/2}\big|_{2,\Omega}^2+c(1/\varepsilon)\intop_\Omega \psi_1^2|v_\varphi|^sdx.
$$
By the Poincar\'e inequality,
$$
|\nabla|v_\varphi|^{s/2}|_{2,\Omega}^2\ge c|v_\varphi|_{3s,\Omega}^s
$$
so we can estimate the second term on the r.h.s. of (\ref{4.28}) by
$$
|f_\varphi|_{{3s\over 2s+1},\Omega}|v_\varphi|_{3s,\Omega}^{s-1}\le\varepsilon_1 |v_\varphi|_{3s,\Omega}^s+c(1/\varepsilon_1)|f_\varphi|_{{3s\over 2s+1},\Omega}^s.
$$
Using the above estimates with sufficiently small $\varepsilon$, $\varepsilon_1$ in (\ref{4.28}) we derive the inequality
\begin{equation}\eqal{
&{1\over s}{d\over dt}|v_\varphi|_{s,\Omega}^s+{1\over s}|\nabla|v_\varphi|^{s/2}|_{2,\Omega}^2+{1\over s}|v_\varphi|_{3s,\Omega}^s\cr
&\le cs\intop_\Omega\psi_1^2|v_\varphi|^sdx+cs|f_\varphi|_{{3s\over 2s+1},\Omega}^s.\cr}
\label{4.29}
\end{equation}
In view of Lemma \ref{l2.3} the first term on the r.h.s. of (\ref{4.29}) is bounded by
$$
cs|u|_{\infty,\Omega^t}^6\intop_\Omega{\psi_1^2\over r^6}|v_\varphi|^{s-6}dx\le csD_2^6|v_\varphi|_{\infty,\Omega}^{s-6}\intop_\Omega{\psi_1^2\over r^6}dx.
$$
Using the estimate in (\ref{4.29}) yields
\begin{equation}
{1\over s}{d\over dt}|v_\varphi|_{s,\Omega}^s\le csD_2^6|v_\varphi|_{\infty,\Omega}^{s-6}\intop_\Omega{\psi_1^2\over r^6}dx+cs|f_\varphi|_{{3s\over 2s+1},\Omega}^s.
\label{4.30}
\end{equation}
Integrating (\ref{4.30}) with respect to time and using Lemma 3.1 from \cite{NZ}, we obtain
\begin{equation}\eqal{
|v_\varphi|_{s,\Omega}^s&\le cs^2D_2^6|v_\varphi|_{\infty,\Omega^t}^{s-6}\|\Gamma\|_{1,2,\Omega^t}^2\cr
&\quad+cs^2|f_\varphi|_{{3s\over 2s+1},s,\Omega^t}^s+|v_\varphi(0)|_{s,\Omega}^s\cr
&\equiv cs^2D_2^6|v_\varphi|_{\infty,\Omega^t}^{s-6}\|\Gamma\|_{1,2,\Omega^t}^2+cD_9^s(s)\cr}
\label{4.31}
\end{equation}
Dividing (\ref{4.31}) by $|v_\varphi|_{\infty,\Omega^t}^{s-6}$ implies
\begin{equation}
\bigg|{|v_\varphi|_{s,\infty,\Omega^t}\over|v_\varphi|_{\infty,\Omega^t}}\bigg|^{s-6} |v_\varphi|_{s,\Omega}^6\le cs^2D_2^6\|\Gamma\|_{1,\Omega^t}^2+ {c\over|v_\varphi|_{\infty,\Omega^t}^{s-6}}D_9^s(s).
\label{4.32}
\end{equation}
The dividing by $|v_\varphi|_{\infty,\Omega^t}$ is justified because the following two cases are exluded from this paper:

\begin{equation}\eqal{
&\textrm{In the case}\ v_\varphi=0\ \textrm{the existence of global regular solutions to }\cr
&\textrm{problem (\ref{1.6}) is proved in \cite{L1, L2, UY}}.}
\label {4.33}
\end{equation}
\begin{equation}\eqal{
&\textrm{The existence of global regular solutions to problem (\ref{1.6}) for}\cr &v_\varphi\ \textrm{sufficiently small is proved in Appendix B}.}
\label {4.34}
\end{equation}
Since cases (\ref{4.33}) and (\ref{4.34}) are not considered in this paper we can show existence of positive constants $c_0$ and $c_1$ such that
\begin{equation}
{|v_\varphi|_{s,\infty,\Omega^t}\over|v_\varphi|_{\infty,\Omega^t}}\ge\bar c_0
\label{4.35}
\end{equation}
and
\begin{equation}
{1\over|v_\varphi|_{\infty,\Omega^t}}\le\bar c_1.
\label{4.36}
\end{equation}
In view of (\ref{4.35}) and (\ref{4.36}) inequality (\ref{4.32}) takes the form
\begin{equation}
\bar c_0|v_\varphi|_{s,\infty,\Omega^t}^6\le cs^2\|\Gamma\|_{1,2,\Omega^t}^2+c\bar c_1D_9^s(s).
\label{4.37}
\end{equation}
Let $d=12$. Then $\theta={1\over 4}(3\varepsilon_1-\varepsilon_2)$ and (\ref{4.23}) for $d=12$ takes the form
\begin{equation}
X^2\le c_0|v_\varphi|_{12,\infty,\Omega^t}^{16\varepsilon\over 3\varepsilon_1-\varepsilon_2}(1+|v_\varphi|_{\infty,\Omega^t}^{2\varepsilon_0})+ c_0|v_\varphi|_{12,\infty,\Omega^t}^{8\varepsilon\over 3\varepsilon_1-\varepsilon_2}+D_8.
\label{4.38}
\end{equation}
Taking (\ref{4.37}) for $s=12$ and using (\ref{4.38}) yield
\begin{equation}\eqal{
|v_\varphi|_{12,\infty,\Omega^t}^6&\le c_0|v_\varphi|_{12,\infty,\Omega^t}^{16\varepsilon\over 3\varepsilon_1-\varepsilon_2}(1+|v_\varphi|_{\infty,\Omega^t}^{2\varepsilon_0})\cr
&\quad+c_0|v_\varphi|_{12,\infty,\Omega^t}^{8\varepsilon\over 3\varepsilon_1-\varepsilon_2}+cD_8+cD_9^{12}.\cr}
\label{4.39}
\end{equation}
To derive any estimate from (\ref{4.39}) we need
\begin{equation}
{16\varepsilon\over 3\varepsilon_1-\varepsilon_2}<6
\label{4.40}
\end{equation}
We see that (\ref{4.40}) holds for
\begin{equation}
\varepsilon_1> 11\varepsilon_2
\label{4.41}
\end{equation}
In view of the Young inequality, (\ref{4.39}) implies
\begin{equation}
|v_\varphi|_{12,\infty,\Omega^t}^6\le c|v_\varphi|_{\infty,\Omega^t}^{{6(3\varepsilon_1-\varepsilon_2)\over \varepsilon_1-11\varepsilon_2}\varepsilon_0}+c+c(D_8+D_9^{12})
\label{4.42}
\end{equation}
The above inequality implies (\ref{4.27}) and concludes the proof.
\end{proof}

\begin{remark}\label{r4.6}
Exploiting (\ref{4.42}) in (\ref{4.38}) implies the inequality
\begin{equation}
X^2\le c(1+|v_\varphi|_{\infty,\Omega^t}^{2\varepsilon_0}) |v_\varphi|_{\infty,\Omega^t}^{{96\varepsilon\over \varepsilon_1-11\varepsilon_2}\varepsilon_0}+\phi(D_5,D_7,D_8,D_9),
\label{4.43}
\end{equation}
where $X$ is introduced in (\ref{4.22}).
\end{remark}

To prove Theorem \ref{t1.2} we need an estimate for $|v_\varphi|_{\infty,\Omega^t}$. For this purpose we need the result

\begin{lemma}\label{l4.7}
Assume that quantities $D_2$, $D_5$, $D_7$, $D_8$, $D_9$ are bounded. Assume that $f_\varphi/r\in L_1(0,t;L_\infty(\Omega))$, $v_\varphi(0)\in L_\infty(\Omega)$.\\
Then there exists an increasing positive function $\phi$ such that
\begin{equation}
\|v_\varphi\|_{\infty,\Omega^t}\le\phi(D_2,D_5,D_7,D_8,D_9, \|f_\varphi/r\|_{L_1(0,t;L_\infty(\Omega))},|v_\varphi(0)|_{\infty,\Omega}).
\label{4.44}
\end{equation}
\end{lemma}

\begin{proof}
Recall equation $(\ref{1.7})_2$ for $v_\varphi$
\begin{equation}
v_{\varphi,t}+v\cdot\nabla v_\varphi-\nu\bigg(\Delta v_\varphi-{1\over r^2}v_\varphi\bigg)=\psi_{1,z}v_\varphi+f_\varphi,
\label{4.45}
\end{equation}
where ${v_r\over r}=-\psi_{1,z}$.

Multiplying (\ref{4.45}) by $v_\varphi|v_\varphi|^{s-2}$ and integrating over $\Omega$ yields
\begin{equation}\eqal{
&{1\over s}{d\over dt}|v_\varphi|_{s,\Omega}^s+{4\nu(s-1)\over s^2}|\nabla|v_\varphi|^{s/2}|_{2,\Omega}^2+\nu\intop_\Omega{|v_\varphi|^s\over r^2}dx\cr
&=\intop_\Omega\psi_{1,z}v_\varphi^2|v_\varphi|^{s-2}dx+\intop_\Omega f_\varphi v_\varphi|v_\varphi|^{s-2}dx.\cr}
\label{4.46}
\end{equation}
The first term on the r.h.s. of (\ref{4.46}) is bounded by
$$
\intop_\Omega|\psi_{,z}|\,|v_\varphi|^{s/2}{|v_\varphi|^{s/2}\over r}dx\le\varepsilon\intop_\Omega{|v_\varphi|^s\over r^2}dx+c(1/\varepsilon)\intop_\Omega\psi_{,z}^2|v_\varphi|^sdx,
$$
where the second integral is bounded by
$$
|rv_\varphi|_{\infty,\Omega}^2\intop_\Omega|\psi_{1,z}|^2|v_\varphi|^{s-2}dx\le D_2^2|\psi_{1,z}|_{s,\Omega}^2|v_\varphi|_{s,\Omega}^{s-2}.
$$
The second term on the r.h.s. of (\ref{4.46}) is estimated by
$$\eqal{
&\intop_\Omega|f_\varphi|\,|v_\varphi|^{s-1}dx=\intop_\Omega\bigg|{f_\varphi\over r}\bigg|r|v_\varphi|^{s-1}dx\cr
&\le|rv_\varphi|_{\infty,\Omega}\intop_\Omega\bigg|{f_\varphi\over r}\bigg|\,|v_\varphi|^{s-2}dx\le D_2\bigg|{f_\varphi\over r}\bigg|_{s/2,\Omega}|v_\varphi|_{s,\Omega}^{s-2}.\cr}
$$
Using the above estimates in (\ref{4.46}) and assuming that $\varepsilon$ is sufficiently small we obtain the inequality
$$
{1\over s}{d\over dt}|v_\varphi|_{s,\Omega}^s\le D_2^2\bigg(|\psi_{1,z}|_{s,\Omega}^2|v_\varphi|_{s,\Omega}^{s-2}+\bigg|{f_\varphi\over r}\bigg|_{s/2,\Omega}|v_\varphi|_{s,\Omega}^{s-2}\bigg).
$$
Simplifying, we get
$$
{1\over 2}{d\over dt}|v_\varphi|_{s,\Omega}^2\le D_2^2\bigg(|\psi_{1,z}|_{s,\Omega}^2+\bigg|{f_\varphi\over r}\bigg|_{s/2,\Omega}\bigg).
$$
Integrating with respect to time and passing with $s\to\infty$, we derive
\begin{equation}
|v_\varphi(t)|_{\infty,\Omega}^2\le D_2^2\bigg(\intop_0^t|\psi_{1,z}|_{\infty,\Omega}^2dt'+\intop_0^t\bigg|{f_\varphi\over r}\bigg|_{\infty,\Omega}dt'\bigg)+|v_\varphi(0)|_{\infty,\Omega}^2.
\label{4.47}
\end{equation}
Since $\intop_0^t|\psi_{1,z}|_{\infty,\Omega}^2dt'\le X^2$ we can apply (\ref{4.43}). Then(\ref{4.47}) takes the form
\begin{equation}\eqal{
|v_\varphi|_{\infty,\Omega^t}^2&\le D_2^2(1+ |v_\varphi|_{\infty,\Omega^t}^{2\varepsilon_0}) |v_\varphi|_{\infty,\Omega^t}^{{96\varepsilon\over \varepsilon_1-11\varepsilon_2}\varepsilon_0}+D_2\phi(D_5,D_7,D_8,D_9)\cr
&\quad+D_2^2\intop_0^t\bigg|{f_\varphi\over r}\bigg|_{\infty,\Omega}dt'+|v_\varphi(0)|_{\infty,\Omega}^2.\cr}
\label{4.48}
\end{equation}
Hence for $\varepsilon_0$ sufficiently small we derive (\ref{4.44}). this ends the proof.
\end{proof}

\begin{remark}\label{r4.8}
Inequalities (\ref{4.43}) and (\ref{4.44}) imply
\begin{equation}
X\le\phi(D_2,D_5,D_7,D_8,D_9,|f_\varphi/r|_{\infty,1,\Omega^t}, |v_\varphi(0)|_{\infty,\Omega})
\label{4.49}
\end{equation}
The above inequality proves Theorem \ref{t1.2}.
\end{remark}

\section{Estimates for the swirl}\label{s5}

In this Section we find estimates for solutions to the problem
\begin{equation}\eqal{
&u_{,t}+v\cdot\nabla u-\nu\Delta u+2\nu{u_{,r}\over r}=rf_\varphi\equiv f_0\quad &{\rm in}\ \ \Omega^t,\cr
&u|_{S_1}=0,\quad &{\rm in}\ \ \Omega^t,\cr
&u|_{S_2}\ -\ \textrm{periodic boundary conditions},\cr
&u|_{t=0}=u(0)\quad &{\rm in}\ \ \Omega.\cr}
\label{5.1}
\end{equation}

\begin{lemma}\label{l5.1}
Assume that $D_1$, $D_2$ are described by (\ref{2.1}) and (\ref{2.7}), respectively. Let $u_{,z}(0),u_{,r}(0)\in L_2(\Omega)$, $f_0\in L_2(\Omega^t)$.\\
Then solutions to (\ref{5.1}) satisfy the estimates
\begin{equation}
|v_{,z}(t)|_{2,\Omega}^2+\nu|\nabla u_{,z}|_{2,\Omega^t}^2\le c(D_1^2D_2^2+|u_{,z}(0)|_{2,\Omega}^2+|f_0|_{2,\Omega^t}^2)\equiv cD_3^2,
\label{5.2}
\end{equation}
\begin{equation}\eqal{
&|u_{,r}(t)|_{2,\Omega}^2+\nu(|u_{,rr}|_{2,\Omega^t}^2+|u_{,rz}|_{2,\Omega^t}^2)\le cD_1^2(1+D_2^2)\cr
&\quad+|u_{,r}(0)|_{2,\Omega}^2+|f_0|_{2,\Omega^t}^2+|f_0|_{4/3,2,S_1^t}^2\equiv cD_4^2.\cr}
\label{5.3}
\end{equation}
\end{lemma}

\begin{proof}
Differentiate (\ref{5.1}) with respect to $z$, multiply by $u_{,z}$ and integrate over $\Omega$. To apply the Green theorem we have to consider problem (\ref{5.1}) in domain $\bar\Omega=\{x\in\R^3\colon r<R,z\in(-a,a),\varphi\in(0,2\pi)\}$. Then we obtain
\begin{equation}\eqal{
&{1\over 2}{d\over dt}|u_{,z}|_{2,\bar\Omega}^2-\nu\intop_{\bar\Omega}\divv(\nabla u_{,z}u_{,z})d\bar x+\nu\intop_{\bar\Omega}|\nabla u_{,z}|^2d\bar x\cr
&=-\intop_{\bar\Omega}v_{,z}\cdot\nabla u\cdot u_{,z}d\bar x+\intop_{\bar\Omega}f_{0,z}u_{,z}d\bar x,\cr}
\label{5.4}
\end{equation}
where $d\bar x=dxd\varphi$. The second term on the l.h.s. implies a boundary term which vanishes in view boundary conditions. Since all functions in (\ref{5.4}) do not depend on $\varphi$ any integral with respect to $\varphi$ can be dropped.

Integrating by parts with respect to $z$ in the term from the r.h.s. of (\ref{5.4}) and using the boundary conditions on $S_2$ we obtain
\begin{equation}\eqal{
&{1\over 2}{d\over dt}|u_{,z}|_{2,\Omega}^2+\nu|\nabla u_{,z}|_{2,\Omega}^2-{\nu\over 2}\intop_{-a}^au_{,z}^2\bigg|_{r=0}^{r=R}dz\cr
&=\intop_\Omega v_{,zz}\cdot\nabla uudx+\intop_\Omega v_{,z}\cdot\nabla u_{,z}\cdot udx-\intop_\Omega f_0u_{,zz}dx.\cr}
\label{5.5}
\end{equation}
The last term on the l.h.s. of (\ref{5.5}) vanishes and the first term on the r.h.s. equals
$$
{1\over 2}\intop_\Omega v_{,zz}\cdot\nabla u^2dx={1\over 2}\intop_{S_1}v_{,zz}\cdot\bar nu^2dS_1=0.
$$
Applying the H\"older and Young inequalities to the other terms from the r.h.s. of (\ref{5.5}) yields
\begin{equation}
{d\over dt}|u_{,z}|_{2,\Omega}^2+\nu|\nabla u_{,z}|_{2,\Omega}^2\le c|u|_{\infty,\Omega}^2|v_{,z}|_{2,\Omega}^2+c|f_0|_{2,\Omega}^2.
\label{5.6}
\end{equation}
Integrating (\ref{5.6}) with respect to time gives
\begin{equation}\eqal{
&|u_{,z}(t)|_{2,\Omega}^2+\nu|\nabla u_{,z}|_{2,\Omega}^2\le c|u|_{\infty,\Omega^t}^2|v_{,z}|_{2,\Omega^t}^2+|u_{,z}(0)|_{2,\Omega}^2\cr
&\quad+c|f_0|_{2,\Omega^t}^2\le cD_1^2D_2^2+|u_{,z}(0)|_{2,\Omega}^2+c|f_0|_{2,\Omega^t}^2.\cr}
\label{5.7}
\end{equation}
The above inequality implies (\ref{5.2}).

Differentiating (\ref{5.1}) with respect to $r$ gives
\begin{equation}
u_{,rt}+v\cdot\nabla u_{,r}+v_{,r}\cdot\nabla u-\nu(\Delta u)_{,r}+{2\nu\over r}u_{,rr}-{2\nu\over r^2}u_{,r}=f_{0,r}.
\label{5.8}
\end{equation}
Multiplying (\ref{5.8}) by $u_{,r}$ and integrating over $\Omega$ yields
\begin{equation}\eqal{
&{1\over 2}{d\over dt}|u_{,r}|_{2,\Omega}^2+\intop_\Omega v_{,r}\cdot\nabla uu_{,r}dx-\nu\intop_\Omega(\Delta u)_{,r}u_{,r}dx\cr
&\quad+2\nu\intop_\Omega{1\over r}u_{,rr}u_{,r}dx-2\nu\intop_\Omega{u_{,r}^2\over r^2}dx=\intop_\Omega f_{0,r}u_{,r}dx.\cr}
\label{5.9}
\end{equation}
Now, we examine the particular terms in (\ref{5.9}). The second term equals
$$\eqal{
&\intop_\Omega v_{,r}\cdot\nabla uu_{,r}rdrdz=\intop_\Omega(v_{r,r}\partial_ru+ v_{z,r}\partial_zu)u_{,r}rdrdz\cr
&=\intop_\Omega(rv_{r,r}u_{,r}+rv_{z,r}u_{,z})u_{,r}drdz\cr
&=\intop_\Omega(rv_{r,r}u_{,r}u_{,r}+rv_{z,r}u_{,r}u_{,z})drdz\cr
&=-\intop_\Omega[(rv_{r,r}u_{,r})_{,r}+(rv_{z,r}u_{,r})_{,z}]udrdz\equiv I,\cr}
$$
where we used that $rv_{r,r}u_{,r}u|_{r=0}=0$ (see \cite{LW}). Continuing, we write $I$ in the form
$$\eqal{
I&=-\intop_\Omega[(rv_{r,r})_{,r}+(rv_{z,r})_{,z}]u_{,r}udrdz\cr
&\quad-\intop_\Omega[rv_{r,r}u_{,rr}+rv_{z,r}u_{,rz}]udrdz\equiv I_1+I_2.\cr}
$$
To estimate $I_1$, we calculate
$$
I_1^1=(rv_{r,r})_{,r}+(rv_{z,r})_{,z}=rv_{r,rr}+v_{r,r}+rv_{z,rz}.
$$
Since $v=v_r\bar e_r+v_z\bar e_z$ is divergence free, we have
\begin{equation}
v_{r,r}+v_{z,z}+{v_r\over r}=0.
\label{5.10}
\end{equation}
Since equation (\ref{5.10}) is satisfied identically in $\Omega$, we can differentiate (\ref{5.10}) with respect to $r$. Then, we get
$$
v_{r,rr}+v_{z,zr}+{v_{r,r}\over r}-{v_r\over r^2}=0.
$$
Hence
$$
I_1^1={v_r\over r}.
$$
Then, $I_1$ equals
$$
I_1=-\intop_\Omega{v_r\over r}u_{,r}udrdz.
$$
Therefore
\begin{equation}
\bigg|\intop_0^tI_1dt'\bigg|\le\bigg|{v_r\over r}\bigg|_{2,\Omega^t}\bigg|{u_{,r}\over r}\bigg|_{2,\Omega^t}|u|_{\infty,\Omega^t}.
\label{5.11}
\end{equation}
Next
$$
|I_2|\le\varepsilon(|u_{,rr}|_{2,\Omega}^2+|u_{,rz}|_{2,\Omega}^2)+c(1/\varepsilon) |u|_{\infty,\Omega}^2(|v_{r,r}|_{2,\Omega}^2+|v_{z,r}|_{2,\Omega}^2).
$$
The third integral in (\ref{5.9}) equals
$$\eqal{
J&=-\nu\intop_\Omega(\Delta u)_{,r}u_{,r}dx=-\nu\intop_\Omega\bigg(u_{,rrr}+\bigg({1\over r}u_{,r}\bigg)_{,r}+u_{,rzz}\bigg)u_{,r}rdrdz\cr
&=-\nu\intop_\Omega\bigg[\bigg(u_{,rr}+{1\over r}u_{,r}\bigg)u_{,r}r\bigg]_{,r}drdz+
\nu\intop_\Omega u_{,rr}(u_{,r}r)_{,r}drdz\cr
&\quad+\nu\intop_\Omega{1\over r}u_{,r}(u_{,r}r)_{,r}drdz+\intop_\Omega u_{,rz}^2dx=-\nu\intop_{-a}^a\bigg(u_{,rr}+{1\over r}u_{,r}\bigg)u_{,r}r\bigg|_{r=0}^{r=R}dz\cr
&\quad+\nu\intop_\Omega(u_{,rr}^2+u_{,rz}^2)dx+\nu\intop_\Omega{u_{,r}^2\over r^2}dx+2\nu\intop_\Omega u_{,rr}u_{,r}drdz,\cr}
$$
where the last term equals
\begin{equation}
\nu\intop_\Omega(u_{,r}^2)_{,r}drdz=\nu\intop_{-a}^au_{,r}^2\bigg|_{r=0}^{r=R}dz= \nu\intop_{-a}^au_{,r}^2\bigg|_{r=R}dz
\label{5.12}
\end{equation}
because $u_{,r}|_{r=0}=(v_\varphi+v_{\varphi,r}r)|_{r=0}=0$.
\goodbreak

To examine the boundary term in $J$ we recall the expansion of $v_\varphi$ near the axis of symmetry (see \cite{LW})
$$
v_\varphi=a_1(z,t)r+a_2(z,t)r^3+\cdots,
$$
so
$$
u=a_1(z,t)r^2+a_2(z,t)r^4+\cdots
$$
Then $\big(u_{,rr}+{1\over r}u_{,r}\big)u_{,r}r|_{r=0}=0$ and we have to emphasize that all calculations in this paper are performed for sufficiently regular solutions.

Therefore, the boundary term in $J$ equals
$$
J_1=-\nu\intop_{-a}^a\bigg(u_{,rr}+{1\over r}u_{,r}\bigg)u_{,r}r\bigg|_{r=R}dz.
$$
Projecting $(\ref{5.1})_1$ on $S_1$ yields
$$
-\nu\bigg(u_{,rr}+{1\over r}u_{,r}\bigg)+2\nu{u_{,r}\over r}=f_0\quad {\rm on}\ \ S_1.
$$
Hence
$$
u_{,rr}|_{S_1}=\bigg({u_{,r}\over r}-{1\over\nu}f_0\bigg)\bigg|_{S_1}.
$$
Using the expression in $J_1$ gives
$$
J_1=-2\nu\intop_{-a}^au_{,r}^2\bigg|_{r=R}dz+\intop_{-a}^af_0u_{,r}r\bigg|_{r=R}dz.
$$
The fourth term in (\ref{5.9}) equals (\ref{5.12}).

Using the above estimates and expressions in (\ref{5.9}) yields
\begin{equation}\eqal{
&{1\over 2}{d\over dt}|u_{,r}|_{2,\Omega}^2+\nu\intop_\Omega(u_{,rr}^2+u_{,rz}^2)dx+\nu\intop_\Omega {u_{,r}^2\over r^2}dx\cr
&\quad-2\nu\intop_\Omega{u_{,r}^2\over r^2}dx\le\intop_\Omega\bigg|{v_r\over r}u_{,r}u\bigg|drdz\cr
&\quad+\varepsilon(|u_{,rr}|_{2,\Omega}^2+|u_{,rz}|_{2,\Omega}^2)+c(1/\varepsilon) |u|_{\infty,\Omega}^2(|v_{r,r}|_{2,\Omega}^2+|v_{z,r}|_{2,\Omega}^2)\cr
&\quad+c(1/\varepsilon)|f_0|_{2,\Omega}^2+ \bigg|\intop_{-a}^af_0u_{,r}r\bigg|_{r=R}dz.\cr}
\label{5.13}
\end{equation}
Integrating (\ref{5.13}) with respect to time and assuming that $\varepsilon$ is sufficiently small, we obtain
\begin{equation}\eqal{
&|u_{,r}(t)|_{2,\Omega}^2+\nu(|u_{,rr}|_{2,\Omega^t}^2+|u_{,rz}|_{2,\Omega^t}^2)\le \nu\bigg|{u_{,r}\over r}\bigg|_{2,\Omega^t}^2\cr
&\quad+c\bigg|{v_r\over r}\bigg|_{2,\Omega^t}\bigg|{u_{,r}\over r}\bigg|_{2,\Omega^t}|u|_{\infty,\Omega^t}+c|u|_{\infty,\Omega^t}^2 (|v_{r,r}|_{2,\Omega^t}^2+|v_{z,r}|_{2,\Omega^t}^2)\cr
&\quad+c|f_0|_{2,\Omega^t}^2+|u_{,r}(0)|_{2,\Omega}^2+\nu\intop_0^t\intop_{-a}^a u_{,r}^2\bigg|_{r=R}dxdt'\cr
&\quad+\bigg|\intop_0^t\intop_{-a}^af_0u_{,r}r\bigg|_{r=R}dxdt'\bigg|.\cr}
\label{5.14}
\end{equation}
Using
$$
\intop_{\Omega^t}\bigg|{u_{,r}\over r}\bigg|^2dxdt'\le\intop_{\Omega^t}\bigg( |v_{\varphi,r}|^2+{v_\varphi^2\over r^2}\bigg)dxdt'\le cD_1^2
$$
and
$$
\intop_0^t\intop_{-a}^au_{,r}^2\bigg|_{r=R}dxdt'\le\varepsilon|\nabla u_{,r}|_{2,\Omega^t}^2+c(1/\varepsilon)|u_{,r}|_{2,\Omega^t}^2,
$$
$$\eqal{
&\bigg|\intop_0^t\intop_{-a}^af_0u_{,r}\bigg|_{r=R}dxdt'\le \varepsilon_1|u_{,r}|_{4,2,S_1^t}^2+c(1/\varepsilon_1)|f_0|_{4/3,2,S_1^t}^2\cr
&\le \varepsilon_1(|u_{,rr}|_{2,\Omega^t}^2+|u_{,rz}|_{2,\Omega^t}^2)+c(1/\varepsilon_1) |f_0|_{4/3,2,S_1^t}^2\cr}
$$
and Lemmas \ref{l2.2}, \ref{l2.3} we have
\begin{equation}\eqal{
&|u_{,r}(t)|_{2,\Omega}^2+\nu(|u_{,rr}|_{2,\Omega^t}^2+|u_{,rz}|_{2,\Omega^t}^2)\le c(D_1^2+D_1^2D_2+D_1^2D_2^2)\cr
&\quad+c|f_0|_{2,\Omega^t}^2+c|f_0|_{4/3,2,S_1^t}^2+|u_{,r}(0)|_{2,\Omega}^2.\cr}
\label{5.15}
\end{equation}
This inequality implies (\ref{5.3}) and concludes the proof.
\end{proof}

\section{Estimates for $\omega_r$, $\omega_z$}\label{s6}

\begin{lemma}\label{l6.1}
Assume that $D_5=D_2(D_1+D_3+D_4)$, $D_6=D_2^{1-\varepsilon_0}D_3$ where $D_1$, $D_2$ are introduced in (\ref{2.1}) and (\ref{2.7}) and, $D_3$, $D_4$ in (\ref{5.2}) (\ref{5.3}), respectively. Let
$$\eqal{
D_7&=|F_r|_{6/5,2,\Omega^t}^2+|F_z|_{6/5,2,\Omega^t}^2+|\omega_r(0)|_{2,\Omega}^2+ |\omega_z(0)|_{2,\Omega}^2\cr
&\quad+|f_\varphi|_{2,S_1^t}(D_3+D_4)<\infty.\cr}
$$
Let $\varepsilon_0$ be arbitrary small positive number and let $v_\varphi\in L_\infty(\Omega^t)$.\\
Let $\Gamma\in L_2(0,t;H^1(\Omega))$.\\
Then
\begin{equation}\eqal{
&\|\omega_r\|_{V(\Omega^t)}^2+\|\omega_z\|_{V(\Omega^t)}^2+|\Phi|_{2,\Omega^t}^2\le cD_5|\Gamma_{,z}|_{2,\Omega^t}\cr
&\quad+cD_6|v_\varphi|_{\infty,\Omega^t}^{\varepsilon_0}\|\Gamma\|_{1,2,\Omega^t}+ cD_7.\cr}
\label{6.1}
\end{equation}
\end{lemma}

\begin{proof}
Multiplying $(\ref{1.9})_1$ by $\omega_r$, $(\ref{1.9})_3$ by $\omega_z$, integrating over $\Omega^t$ and adding yield
\begin{equation}\eqal{
&{1\over 2}(|\omega_r(t)|_{2,\Omega}^2+|\omega_z(t)|_{2,\Omega}^2)+\nu (|\nabla\omega_r|_{2,\Omega^t}^2+|\nabla\omega_z|_{2,\Omega^t}^2)\cr
&\quad+\nu\bigg|{\omega_r\over r}\bigg|_{2,\Omega^t}^2-\nu\intop_{S_1^t}\bar n\cdot\nabla\omega_r\omega_rdS_1dt'-\nu\intop_{S_1^t}\bar n\cdot\nabla\omega_z\omega_zdS_1dt'\cr
&=\intop_{\Omega^t}[v_{r,r}\omega_r^2+v_{z,z}\omega_z^2+(v_{r,z}+v_{z,r}) \omega_r\omega_z]dxdt'\cr
&\quad+\intop_{\Omega^t}(F_r\omega_r+F_z\omega_z)dxdt'+{1\over 2}(|\omega_r(0)|_{2,\Omega}^2+|\omega_z(0)|_{2,\Omega}^2)\cr
&\equiv J+\intop_{\Omega^t}(F_r\omega_r+F_z\omega_z)dxdt'+{1\over 2}(|\omega_r(0)|_{2,\Omega}^2+|\omega_z(0)|_{2,\Omega}^2).\cr}
\label{6.2}
\end{equation}
Since $\omega_r=-v_{\varphi,z}$ and $v_\varphi|_{r=R}=0$ we obtain
$$
-\intop_{S_1}\bar n\cdot\nabla\omega_r\omega_rdS_1=0.
$$
Using $\omega_z=v_{\varphi,r}+{v_\varphi\over r}$ we derive
$$
-\nu\intop_{S_1^t}\bar n\cdot\nabla\omega_z\omega_zdS_1dt'= -\nu\intop_0^t\intop_{-a}^a\partial_r\bigg(v_{\varphi,r}+{v_\varphi\over r}\bigg)\bigg(v_{\varphi,r}+{v_\varphi\over r}\bigg)\bigg|_{r=R}Rdzdt'\equiv I_1.
$$
Since $v_\varphi|_{r=R}=0$ $I_1$ takes the form
$$
I_1=-\nu\intop_0^t\intop_{-a}^a\bigg(v_{\varphi,rr}+{v_{\varphi,r}\over r}\bigg)v_{\varphi,r}\bigg|_{r=R}Rdzdt'.
$$
Projecting $(\ref{1.7})_2$ on $S_1$ yields
$$
-\nu\bigg(v_{\varphi,rr}+{1\over r}v_{\varphi,r}\bigg)=f_\varphi\quad {\rm on}\ \ {S_1}.
$$
Hence
\begin{equation}
I_1=R\intop_0^t\intop_{-a}^af_\varphi v_{\varphi,r}\bigg|_{r=R}dzdt'=\intop_0^t\intop_{-a}^af_\varphi\bigg(u_{,r}-{1\over R}u\bigg)\bigg|_{r=R}dzdt'.
\label{6.3}
\end{equation}
Using (\ref{1.13}) and (\ref{1.21}) in $J$ implies
$$\eqal{
J&=\intop_{\Omega^t}\bigg[-{1\over r^2}u_{,z}^2(\psi_{1,z}+r\psi_{1,rz})+\bigg({1\over r}u_{,r}\bigg)^2(r\psi_{1,zr}+2\psi_{1,z})\cr
&\quad-{1\over r^2}u_{,r}u_{,z}(-r\psi_{1,zz}+3\psi_{1,r}+r\psi_{1,rr})\bigg]dxdt'\equiv J_1+J_2+J_3.\cr}
$$
We integrate by parts in $J_1$ and use the boundary conditions on $S_2$. Then we have
$$\eqal{
J_1&=\intop_{\Omega^t}{1\over r^2}uu_{,zz}(\psi_{1,z}+r\psi_{1,rz})dxdt'+\intop_{\Omega^t}{1\over r^2}uu_{,z}(\psi_{1,zz}+r\psi_{1,rzz})dxdt'.\cr}
$$
Now, we estimate the particular terms in $J_1$,
$$\eqal{
J_{11}&=\bigg|\intop_{\Omega^t}uu_{,zz}{1\over r}\psi_{1,rz}dxdt'\bigg|\le |u|_{\infty,\Omega^t}|u_{,zz}|_{2,\Omega^t}\bigg|{1\over r}\psi_{1,rz}\bigg|_{2,\Omega^t},\cr
J_{12}&=\bigg|\intop_{\Omega^t}u{u_{,z}\over r}\psi_{1,rzz}dxdt'\bigg|\le |u|_{\infty,\Omega^t}\bigg|{u_{,z}\over r}\bigg|_{2,\Omega^t}|\psi_{1,rzz}|_{2,\Omega^t},\cr
J_{13}&=\bigg|\intop_{\Omega^t}u{u_{,z}\over r}{\psi_{1,zz}\over r}dxdt'\bigg|\le |u|_{\infty,\Omega^t}\bigg|{u_{,z}\over r}\bigg|_{2,\Omega^t} \bigg|{\psi_{1,zz}\over r}\bigg|_{2,\Omega^t},\cr
J_{14}&=\bigg|\intop_{\Omega^t}{1\over r^2}uu_{,zz}\psi_{1,z}dxdt'\bigg|= \bigg|\intop_{\Omega^t}uu_{,zz}{\psi_{1,z}\over r^2}dxdt'\bigg|\cr
&\le|u|_{\infty,\Omega^t}|u_{,zz}|_{2,\Omega^t}\bigg|{\psi_{1,z}\over r^2}\bigg|_{2,\Omega^t},\cr}
$$
where integration by parts can be performed in view of periodic boundary conditions on $S_2$.

Next, we consider $J_2$,
$$\eqal{
J_2&=\intop_{\Omega^t}{1\over r^2}u_{,r}^2(r\psi_{1,zr}+2\psi_{1,z})rdrdzdt'=\intop_{\Omega^t}{1\over r}u_{,r}^2(r\psi_{1,zr}+2\psi_{1,z})drdzdt'\cr
&=\intop_0^t\intop_{-a}^a\bigg[{1\over r}uu_{,r}(r\psi_{1,zr}+2\psi_{1,z})\bigg] \bigg|_{r=0}^{r=R}dzdt'\cr
&\quad-\intop_{\Omega^t}uu_{,rr}\bigg({1\over r}\psi_{1,zr}+{2\over r^2}\psi_{1,z}\bigg)dxdt'\cr
&\quad-\intop_{\Omega^t}uu_{,r}\bigg(\psi_{1,zrr}-{2\over r^2}\psi_{1,z}+{2\over r}\psi_{1,zr}\bigg)drdzdt',\cr}
$$
where the boundary term for $r=R$ vanishes because $u|_{r=R}=0$. To examine the boundary term at $r=0$ we recall from \cite{LW} the expresions near the axis of symmetry
$$
u=a_1(z,t)r^2+a_2(z,t)r^4+\cdots,
$$
so
$$
u_{,r}=2a_1(z,t)r+4a_2(z,t)r^3+\cdots
$$
Then
$$
{1\over r}uu_{,r}(r\psi_{1,zr}+2\psi_{1,z})\sim cr^2(r\psi_{1,zr}+2\psi_{1,z}).
$$
The above expression vanishes for $r=0$ because $\psi_{1,z}$ is bounded near the axis of symmetry.

Now, we estimate the particular terms in $J_2$,
$$\eqal{
&J_{21}=\bigg|\intop_{\Omega^t}uu_{,rr}{1\over r}\psi_{1,rz}dxdt'\bigg|\le |u|_{\infty,\Omega^t}|u_{,rr}|_{2,\Omega^t}\bigg|{1\over r}\psi_{1,zr}\bigg|_{2,\Omega^t},\cr
&J_{22}=\bigg|\intop_{\Omega^t}uu_{,rr}{1\over r^2}\psi_{1,z}dxdt'\bigg|\le |u|_{\infty,\Omega^t}|u_{,rr}|_{2,\Omega^t}\bigg|{1\over r^2}\psi_{1,z}\bigg|_{2,\Omega^t},\cr
&J_{23}=\bigg|\intop_{\Omega^t}u{u_{,r}\over r}\psi_{1,zrr}dxdt'\bigg|\le |u|_{\infty,\Omega^t}\bigg|{u_{,r}\over r}\bigg|_{2,\Omega^t}|\psi_{1,zrr}|_{2,\Omega^t},\cr
&J_{24}=\bigg|\intop_{\Omega^t}u{u_{,r}\over r}{1\over r^2}\psi_{1,z}dxdt'\bigg|\le|u|_{\infty,\Omega^t}\bigg|{u_{,r}\over r}\bigg|_{2,\Omega^t}\bigg|{1\over r^2}\psi_{1,z}\bigg|_{2,\Omega^t},\cr
&J_{25}=\bigg|\intop_{\Omega^t}u{u_{,r}\over r}{1\over r}\psi_{1,zr}dxdt'\bigg|\le|u|_{\infty,\Omega^t}\bigg|{u_{,r}\over r}\bigg|_{2,\Omega^t}\bigg|{1\over r}\psi_{1,zr}\bigg|_{2,\Omega^t}.\cr}
$$
Finally, we examine $J_3$. Integrating by parts with respect to $z$ and using the periodic boundary conditions on $S_2$, we have
$$\eqal{
J_3&=\intop_{\Omega^t}u{1\over r^2}u_{,rz}(-r\psi_{1,zz}+3\psi_{1,r}+r\psi_{1,rr})dxdt'\cr
&\quad+\intop_{\Omega^t}u{1\over r^2}u_{,r}(-r\psi_{1,zzz}+3\psi_{1,rz}+r\psi_{1,rrz})dxdt'.\cr}
$$
Now, we estimate the particular terms in $J_3$,
$$\eqal{
J_{31}&=\bigg|\intop_{\Omega^t}uu_{,rz}{1\over r}\psi_{1,zz}dxdt'\bigg|\le |u|_{\infty,\Omega^t}|u_{,rz}|_{2,\Omega^t}\bigg|{\psi_{1,zz}\over r}\bigg|_{2,\Omega^t},\cr
J_{32}&=\bigg|\intop_{\Omega^t}u{1\over r^2}u_{,rz}\psi_{1,r}dxdt'\bigg|= \bigg|\intop_{\Omega^t}{u\over r^{\varepsilon_0}}u_{,rz}{\psi_{1,r}\over r^{2-\varepsilon_0}}dxdt'\bigg|\cr
&\le|u|_{\infty,\Omega^t}^{1-\varepsilon_0}|v_\varphi|_{\infty,\Omega^t}^{\varepsilon_0} |u_{,rz}|_{2,\Omega^t}\bigg|{\psi_{1,r}\over r^2}\bigg|_{L_2(0,t;L_{2,\varepsilon_0}(\Omega))},\cr}
$$
where $\varepsilon_0>0$ can be chosen as small as we want. Thus
$$\eqal{
J_{33}&=\bigg|\intop_{\Omega^t}{u\over r^{\varepsilon_0}}u_{,rz}{1\over r^{1-\varepsilon_0}}\psi_{1,rr}dxdt'\bigg|\cr
&\le|u|_{\infty,\Omega^t}^{1-\varepsilon_0}|v_\varphi|_{\infty,\Omega^t}^{\varepsilon_0} |u_{,rz}|_{2,\Omega^t}\bigg|{\psi_{1,rr}\over r^{1-\varepsilon_0}}\bigg|_{2,\Omega^t},\cr
J_{34}&=\bigg|\intop_{\Omega^t}u{u_{,r}\over r}\psi_{1,zzz}dxdt'\bigg|\le |u|_{\infty,\Omega^t}\bigg|{u_{,r}\over r}\bigg|_{2,\Omega^t}|\psi_{1,zzz}|_{2,\Omega^t},\cr
J_{35}&=\bigg|\intop_{\Omega^t}u{u_{,r}\over r}{1\over r}\psi_{1,rz}dxdt'\bigg|\le|u|_{\infty,\Omega^t}\bigg|{u_{,r}\over r}\bigg|_{2,\Omega^t}\bigg|{\psi_{1,rz}\over r}\bigg|_{2,\Omega^t},\cr
J_{36}&=\bigg|\intop_{\Omega^t}u{u_{,r}\over r}\psi_{1,rrz}dxdt'\bigg|\le |u|_{\infty,\Omega^t}\bigg|{u_{,r}\over r}\bigg|_{2,\Omega^t}|\psi_{1,rrz}|_{2,\Omega^t}.\cr}
$$
Summarizing the above estimates, we obtain
$$\eqal{
&|J|\le c|u|_{\infty,\Omega^t}\bigg[(|u_{,zz}|_{2,\Omega^t}+|u_{,zr}|_{2,\Omega^t}+ |u_{,rr}|_{2,\Omega^t})\cdot\cr
&\quad\cdot\bigg(\bigg|{1\over r}\psi_{1,rz}\bigg|_{2,\Omega^t}+\bigg|{1\over r} \psi_{1,zz}\bigg|_{2,\Omega^t}+\bigg|{1\over r^2}\psi_{1,z}\bigg|_{2,\Omega^t}\bigg)\cr
&\quad+\bigg(\bigg|{u_{,r}\over r}\bigg|_{2,\Omega^t}+\bigg|{u_{,z}\over r}\bigg|_{2,\Omega^t}\bigg)\bigg(|\psi_{1,rzz}|_{2,\Omega^t}+|\psi_{1,zrr}|_{2,\Omega^t}\cr
&\quad+|\psi_{1,zzz}|_{2,\Omega^t}+\bigg|{1\over r}\psi_{1,zz}\bigg|_{2,\Omega^t}+\bigg|{1\over r}\psi_{1,zr}\bigg|_{2,\Omega^t}+\bigg|{1\over r^2}\psi_{1,z}\bigg|_{2,\Omega^t}\bigg)\bigg]\cr
&\quad+c|u|_{\infty,\Omega^t}^{1-\varepsilon_0} |v_\varphi|_{\infty,\Omega^t}^{\varepsilon_0}|u_{,rz}|_{2,\Omega^t}\bigg(\bigg| {\psi_{1,rr}\over r}\bigg|_{L_2(0,t;L_{2,\varepsilon_0}(\Omega))}+\bigg|{\psi_{1,r}\over r^2}\bigg|_{L_2(0,t;L_{2,\varepsilon_0}(\Omega))}\bigg).\cr}
$$
Using (\ref{2.7}), (\ref{5.2}), (\ref{5.3}) and the estimates from (\ref{2.1})
$$\eqal{
&\bigg|{u_{,r}\over r}\bigg|_{2,\Omega^t}\le\bigg|{v_\varphi\over r}\bigg|_{2,\Omega^t}+|v_{\varphi,r}|_{2,\Omega^t}\le cD_1,\cr
&\bigg|{u_{,z}\over r}\bigg|_{2,\Omega^t}\le|v_{\varphi,z}|_{2,\Omega^t}\le cD_1\cr}
$$
we obtain the following estimate for $J$,
$$\eqal{
&|J|\le c[D_2(D_3+D_4)+D_1D_2]\bigg(|\psi_{1,rrz}|_{2,\Omega^t}+|\psi_{1,rzz}|_{2,\Omega^t}\cr
&\quad+|\psi_{1,zzz}|_{2,\Omega^t}+\bigg|{1\over r}\psi_{1,rz}\bigg|_{2,\Omega^t}+\bigg|{1\over r}\psi_{1,zz}\bigg|_{2,\Omega^t}+\bigg|{1\over r^2}\psi_{1,z}\bigg|_{2,\Omega^t}\bigg)\cr
&\quad+cD_2^{1-\varepsilon_0}D_3|v_\varphi|_{\infty,\Omega^t}^{\varepsilon_0} \bigg(\bigg|{1\over r}\psi_{1,rr}\bigg|_{L_2(0,t;L_{2,\varepsilon_0}(\Omega))}+\bigg|{1\over r^2}\psi_{1,r}\bigg|_{L_2(0,t;L_{2,\varepsilon_0}(\Omega))}\bigg)\equiv J'.\cr}
$$
From (\ref{3.4}), we have (recall that $\omega_1=\Gamma$)
\begin{equation}
|\psi_{1,rrz}|_{2,\Omega^t}+|\psi_{1,rzz}|_{2,\Omega^t}+|\psi_{1,zzz}|_{2,\Omega^t}\le c|\Gamma_{,z}|_{2,\Omega^t}.
\label{6.4}
\end{equation}
Estimates (\ref{3.22}) and (\ref{3.25}) imply
\begin{equation}
\bigg|{\psi_{1,rz}\over r}\bigg|_{2,\Omega^t}+\bigg|{\psi_{1,zz}\over r}\bigg|_{2,\Omega^t}+\bigg|{\psi_{1,z}\over r^2}\bigg|_{2,\Omega^t}\le c|\Gamma_{,z}|_{2,\Omega^t}.
\label{6.5}
\end{equation}
Finally, (\ref{3.33}) yields
\begin{equation}
\bigg|{1\over r}\psi_{1,rr}\bigg|_{L_2(0,t;L_{2,\varepsilon_0}(\Omega))}+\bigg|{1\over r^2}\psi_{1,r}\bigg|_{L_2(0,t;L_{2,\varepsilon_0}(\Omega))}\le cR^{\varepsilon_0}\|\Gamma\|_{1,2,\Omega^t}
\label{6.6}
\end{equation}
Recall that (\ref{6.5}) is valid for $\psi_1|_{r=0}=0$.

This restriction implies that $v_z|_{r=0}=0$, so it is a strong restriction on solutions proved in this paper.

Using (\ref{6.4})--(\ref{6.6}) in $J'$ yields
$$
J'\le cD_2(D_1+D_3+D_4)|\Gamma_{,z}|_{2,\Omega^t}+cD_2^{1-\varepsilon_0}D_3 |v_\varphi|_{\infty,\Omega^t}^{\varepsilon_0}\|\Gamma\|_{1,2,\Omega^t}.
$$
In view of Lemma \ref{l5.1} the term $I_1$ introduced in (\ref{6.3}) is bounded by
$$
I\le c|f_\varphi|_{2,S_1^t}\|u\|_{2,2,\Omega^t}\le c|f_\varphi|_{2,S_1^t}(D_3+D_4).
$$
Using the estimates in (\ref{6.2}), we obtain
\begin{equation}\eqal{
&\|\omega_r\|_{V(\Omega^t)}+\|\omega_z\|_{V(\Omega^t)}+|\Phi|_{2,\Omega^t}\cr
&\le cD_2(D_1+D_3+D_4)|\Gamma_{,z}|_{2,\Omega^t}+cD_2^{1-\varepsilon_0}D_3 |v_\varphi|_{\infty,\Omega^t}^{\varepsilon_0}\|\Gamma\|_{1,2,\Omega^t}\cr
&\quad+c(|F_r|_{6/5,2,\Omega^t}+|F_z|_{6/5,2,\Omega^t})+c(|\omega_r(0)|_{2,\Omega}\cr
&\quad+|\omega_z(0)|_{2,\Omega})+c|f_\varphi|_{2,S_1^t}(D_3+D_4),\cr}
\label{6.7}
\end{equation}
where we used
$$\eqal{
&\bigg|\intop_\Omega(F_r\omega_r+F_z\omega_z)dxdt'\bigg|\le \varepsilon(|\omega_r|_{6,\Omega}^2+|\omega_z|_{6,\Omega}^2)\cr
&\quad+c(1/\varepsilon)(|F_r|_{6/5,\Omega}^2+|F_z|_{6/5,\Omega}^2).\cr}
$$
Hence (\ref{6.7}) implies (\ref{6.1}) and concludes the proof.
\end{proof}

\renewcommand{\theequation}{B.\arabic{equation}}
\setcounter{equation}{0}

\section*{B. Existence of regular solutions to (\ref{1.1}) for\\ \hbox to19pt{} small data}\label{sB}

Recall the quantities
\begin{equation}
u_1={v_\varphi\over r},\ \ \omega_1={\omega_\varphi\over r},\ \ \psi_1={\psi\over r},\ \ f_1={f_\varphi\over r},\ \ F_1={F_\varphi\over r}.
\label{B.1}
\end{equation}
In view of \cite{HL} system (\ref{1.1}) is equivalent to the following one
\begin{equation}\eqal{
&u_{1,t}+v\cdot\nabla u_1-\nu\bigg(\Delta u_1+{2\over r}u_{1,r}\bigg)=2u_1\psi_{1,z}+f_1,\cr
&\omega_{1,t}+v\cdot\nabla\omega_1-\nu\bigg(\Delta\omega_1+{2\over r}\omega_{1,r}\bigg)=2u_1u_{1,z}+F_1,\cr
&-\Delta\psi_1-{2\over r}\psi_{1,r}=\omega_1,\cr
&\textrm{periodic boundary condtions on}\ S_2,\cr
&u_1|_{t=0}=u_1(0),\cr
&\omega_1|_{t=0}=\omega_1(0).\cr}
\label{B.2}
\end{equation}
Functions $u_1$, $\omega_1$, $\psi_1$ have compact support with respect to variable $r$.

Multiplying $(\ref{B.2})_1$ by $u_1|u_1|^2$, integrating over $\Omega$ and using boundary conditions yield
\begin{equation}
{d\over dt}|u_1|_{4,\Omega}^4+\nu|u_1|_{4,\Omega}^4\le c|\omega_1|_{2,\Omega}^2|u_1|_{4,\Omega}^4+c|f_1|_{4,\Omega}^4.
\label{B.3}
\end{equation}
Multiply $(\ref{B.2})_2$ by $\omega_1$, integrate over $\Omega$ and exploit boundary conditions. Then we have
\begin{equation}
{d\over dt}|\omega_1|_{2,\Omega}^2+\nu|\omega_1|_{2,\Omega}^2\le c|u_1|_{4,\Omega}^4+c|F_1|_{2,\Omega}^2.
\label{B.4}
\end{equation}
Introduce the quantity
\begin{equation}
X(t)=|u_1(t)|_{4,\Omega}^4+|\omega_1(t)|_{2,\Omega}^2.
\label{B.5}
\end{equation}
Then (\ref{B.3}) and (\ref{B.4}) imply
\begin{equation}
{d\over dt}X+\nu X\le c_0X^2+G(t),
\label{B.6}
\end{equation}
where
\begin{equation}
G(t)=c(|f_1(t)|_{4,\Omega}^4+|F_1(t)|_{2,\Omega}^2).
\label{B.7}
\end{equation}
Consider (\ref{B.6}) on the time interval $(0,T)$. Assume that for $t\in(0,T)$ the following inequality holds
\begin{equation}
G(t)\le k_0.
\label{B.8}
\end{equation}
Then (\ref{B.6}) takes the form
\begin{equation}
{d\over dt}S\le c_0k_0\bigg({1\over k_0}X^2+1\bigg).
\label{B.9}
\end{equation}
Let $X=\alpha X'$. Then
$$
{d\over dt}X'\le{c_0k_0\over\alpha}\bigg({\alpha^2\over k_0}X'^2+1\bigg).
$$
Setting $\alpha^2=k_0$ yields
\begin{equation}
{d\over dt}X'\le c_0\sqrt{k_0}(X'^2+1).
\label{B.10}
\end{equation}
Integrating (\ref{B.10}) with respect to time implies
$$
\arctg X'(t)-\arctg X'(0)\le c_0\sqrt{k_0}t.
$$
Hence
$$
X'(t)\le{\tg(c_0\sqrt{k_0}t)+X'(0)\over 1-X'(0)\tg(c_0\sqrt{k_0}t)}.
$$
Recalling that $X'(0)={X(0)\over\sqrt{k_0}}$ and setting $y=c_0\sqrt{k_0}t$ we obtain
\begin{equation}
X(t)\le{(\tg y+{X(0)\over\sqrt{k_0}})\sqrt{k_0}\over 1-X(0)c_0t{\tg y\over y}}\le \beta(T),
\label{B.11}
\end{equation}
where $t\le T$. Hence for $T$ large (\ref{B.11}) holds for sufficiently small $X(0)$ and $k_0$.

Consider (\ref{B.6}) in the interval $(0,T)$. Using (\ref{B.11}) we can write (\ref{B.6}) in the form
\begin{equation}
{d\over dt}X+\nu_*X\le G(t),
\label{B.12}
\end{equation}
where$\nu_*=\nu-c_0\beta$.

Integrating (\ref{B.12}) with respect to time yields
\begin{equation}
X(t)\le e^{-\nu_*t}\intop_0^tG(t')e^{\nu_*t'}dt'+e^{-\nu_8t}X(0).
\label{B.13}
\end{equation}
Setting $t=T$ implies
\begin{equation}
X(T)\le\intop_0^TG(t)dt+e^{-\nu_*T}X(0).
\label{B.14}
\end{equation}
For sufficiently small $X(0)$, $k_0$ the time interval $(0,T)$ can be choosen large. Then (\ref{B.14}) can imply that
\begin{equation}
X(T)\le X(0).
\label{B.15}
\end{equation}
Therefore, the previous considerations can be performed for any time interval $(kT,(k+1)T)$, $k\in\N$.

\bibliographystyle{amsplain}

\end{document}